\documentclass[11pt]{amsart}%
\usepackage{amscd,amssymb,verbatim}
\usepackage{amsmath}
\usepackage{amsfonts}
\usepackage{amssymb}
\usepackage{graphicx}%
\theoremstyle{plain}
\newtheorem{thm}{Theorem}

\newtheorem{lem}[thm]{Lemma}
\newtheorem{prop}[thm]{Proposition}

\newtheorem{ex}[thm]{Example}

\begin{document}
\title{Banach-Stone Theorems for maps preserving common zeros}

\begin{abstract}
Let $X$ and $Y$ be completely regular spaces and $E$ and $F$ be Hausdorff
topological vector spaces. We call a linear map $T$ from a subspace of
$C(X,E)$ into $C(Y,F)$ a \emph{Banach-Stone map} if it has the form $Tf(y) =
S_{y}(f(h(y))$  for a family of linear operators $S_{y} : E \to F$, $y \in Y$,
and a function $h: Y \to X$. In this paper, we consider maps having the
property:
\begin{equation}
\cap^{k}_{i=1}Z(f_{i}) \neq\emptyset\iff\cap^{k}_{i=1}Z(Tf_{i}) \neq
\emptyset,\tag{Z}%
\end{equation}
where $Z(f) = \{f  = 0\}$. We characterize linear bijections with property
(Z) between spaces of continuous functions, respectively, spaces of
differentiable functions (including $C^{\infty}$), as Banach-Stone maps. In
particular, we confirm a conjecture of Ercan and \"{O}nal:

Suppose that $X$ and $Y$ are realcompact spaces and $E$ and $F$ are Hausdorff
topological vector lattices (respectively, $C^{*}$-algebras). Let $T: C(X,E)
\to C(Y,F)$ be a vector lattice isomorphism (respectively, $*$-algebra
isomorphism) such that
\[
Z(f) \neq\emptyset\iff Z(Tf) \neq\emptyset.
\]
Then $X$ is homeomorphic to $Y$ and $E$ is lattice isomorphic (respectively,
$C^{*}$-isomorphic) to $F$.

Some results concerning the continuity of $T$ are also obtained.

\end{abstract}
\author{Denny H. Leung}
\address{Department of Mathematics, National University of Singapore, 2 Science Drive
2, Singapore 117543.}
\email{matlhh@nus.edu.sg}
\author{Wee-Kee Tang}
\address{Mathematics and Mathematics Education, National Institute of Education \\
Nanyang Technological University, 1 Nanyang Walk, Singapore 637616.}
\email{weekee.tang@nie.edu.sg}
\thanks{Research of the first author was partially supported by AcRF project no.\ R-146-000-086-112}
\maketitle


\section{Introduction}

Let $X$ and $Y$ be compact Hausdorff spaces. A linear map $T:C(X)\rightarrow
C(Y)$ is a surjective isometry if and only if it has the form $Tf=w\cdot
(f\circ h)$ for some homeomorphism $h:Y\rightarrow X$ and some function $w\in
C(Y)$ such that $|w(y)|=1$ for all $y$. This classical result, called the
Banach-Stone Theorem, was obtained by Banach \cite{B} for compact metric
spaces and extended by Stone \cite{St} to compact Hausdorff spaces. Variations
of this result were obtained by Gelfand and Kolmogorov \cite{GK} for algebraic
isomorphisms and Kaplansky \cite{K} for vector lattice isomorphisms. There has
been a great deal of development in this area in the intervening period,
including extensions to spaces of vector-valued functions. Let us mention in
particular the work on Behrends \cite{Be} on $M$-structures and the
Banach-Stone property and the theory of separating and biseparating maps (see,
e.g., \cite{A1,A2,A2-1,A3, ABN, GJW, JW, LW}). We refer to the article \cite{GJ2}
for a survey on the many results and research areas arising from the
Banach-Stone Theorem.

In a recent article, Ercan and \"{O}nal obtained a Banach lattice-valued
version of the Banach-Stone Theorem.

\begin{thm}
\emph{(Ercan and \"{O}nal \cite{EO})} Let $X$ and $Y$ be compact Hausdorff
spaces and $E$ and $F$ be Banach lattices. If there is an onto vector lattice
isomorphism $T: C(X,E) \to C(Y,F)$ such that
\begin{equation}
Z(f) = \{f = 0\} \neq\emptyset\iff Z(Tf) \neq\emptyset,\tag{Z$_0$}%
\end{equation}
and $F$ is an $AM$-space with unit, then $X$ is homeomorphic to $Y$ and $E$
and $F$ are isomorphic as Banach lattices.
\end{thm}

They further conjectured that the result holds for any Banach lattice $F$. It
is easily observed that for a vector lattice isomorphism, condition (Z$_{0}$)
is equivalent to
\begin{equation}
\cap_{i=1}^{k}Z(f_{i})\neq\emptyset\Longleftrightarrow\cap_{i=1}^{k}%
Z(Tf_{i})\neq\emptyset\tag{Z}%
\end{equation}
for any finite collection of functions $(f_{i})_{i=1}^{k}$. Maps satisfying
(Z) are termed \emph{maps preserving common zeros}. In this paper, we
undertake a general investigation of linear operators preserving common zeros
mapping between certain spaces of vector-valued functions. As a result, we are
able to confirm the conjecture of Ercan and \"{O}nal in a general setting. We
also obtain Banach-Stone type results for mappings between spaces of
vector-valued differentiable functions and some theorems on the automatic
continuity of such maps.  Professor N.\ C.\ Wong has informed us that he and his co-authors have independently solved the conjecture of Ercan and \"{O}nal \cite{CCW}.  The conjecture was also solved by its proposers Ercan and \"{O}nal \cite{EO2}.

In \S 2, there is a study of maps preserving common zeros under rather general
conditions. Theorem \ref{thm2} gives a description of such maps, which shows
that they are almost of Banach-Stone form. In \S 3, we specialize to the case
of vector-valued continuous functions, particularly on realcompact spaces. In
this case, Theorem \ref{thm4} provides a complete characterization of maps
preserving common zeros. The conjecture of Ercan and \"{O}nal \cite{EO}
follows as a corollary. In \S 4, we study maps preserving common zeros mapping
between spaces of vector-valued differentiable functions. We are able to
characterize such maps as Banach-Stone maps under the general assumption that
$E$ and $F$ are Hausdorff topological vector spaces (Theorem \ref{thm8}). Our
result also holds for $C^{\infty}$-functions. In the final section, we
investigate the continuity properties of the maps considered in \S 4. Even
without assuming completeness of $E$ and $F$, continuity of the associated map
$\Phi$ can be obtained for maps between $C^{m}$-spaces when $m \in{\mathbb{N}%
}$. However, this is not longer true if $m = \infty$. To obtain full
continuity of the map $T$ requires completeness of $E$ and $F$. These results,
summarized in Theorem \ref{Unified} and Examples \ref{Exm} and
\ref{Example_infinity}, serve to clarify the role played by completeness in
theorems regarding automatic continuity.

\section{Maps preserving common zeros}

\label{sec1}

If $X$ is a (Hausdorff) completely regular space and $E$ is a Hausdorff
topological vector space, let $C(X,E)$ be the set of all continuous functions
from $X$ into $E$. It is known that $E$ is completely regular \cite{S}. In
particular, every $f\in C(X,E)$ has a unique continuous extension $f^{\beta
}:\beta X\rightarrow\beta E$, where $\beta X$ and $\beta E$ are the
Stone-\u{C}ech compactifications of $X$ and $E$ respectively. The space
$C(X,{\mathbb{R}})$ or $C(X,{\mathbb{C}}),$ as the case may be, is abbreviated
to $C(X)$. A vector subspace $A(X)$ of $C(X)$ is said to be \emph{almost
normal} if for every pair of subsets $P$ and $Q$ of $X$ such that
$\overline{P}^{\beta X}\cap\overline{Q}^{\beta X}=\emptyset$, there exists $f$
in $A(X)$ such that $f(P)\subseteq\{0\}$ and $f(Q)\subseteq\{1\}$. A vector
subspace $A(X,E)$ of $C(X,E)$ is said to be \emph{almost normally
multiplicative} if $A(X,E)$ contains the constant functions and there is an
almost normal subspace $A(X)$ of $C(X)$ so that $\varphi\cdot f\in A(X,E)$
whenever $\varphi\in A(X)$ and $f\in A(X,E)$. For any $f\in C(X,E)$, let
$Z(f)=\{x\in X:f(x)=0\}$ be the \emph{zero set of} $f$. If $u\in E$, let
$\mathbf{u}$ be the function in $C(X,E)$ with constant value $u$.

In this paper, $X$ and $Y$ will always denote (Hausdorff) completely regular
spaces and $E$ and $F$ (nontrivial) Hausdorff topological vector spaces. If
$A(X,E)$ and $A(Y,F)$ are subspaces of $C(X,E)$ and $C(Y,F)$ respectively, a
linear map $T:A(X,E)\rightarrow A(Y,F)$ is said to \emph{preserve common
zeros} if for any $k\in{\mathbb{N}}$ and any sequence $(f_{i})_{i=1}^{k}$ in
$A(X,E)$,
\begin{equation}
\cap_{i=1}^{k}Z(f_{i})\neq\emptyset\Longleftrightarrow\cap_{i=1}^{k}%
Z(Tf_{i})\neq\emptyset. \tag{Z}%
\end{equation}

For the rest of the section, $A\left(  X,E\right)  $ and $A\left(  Y,F\right)
$ will denote almost normally multiplicative vector subspaces of $C\left(
X,E\right)  $ and $C\left(  Y,F\right)  $ respectively and $A\left(  X\right)
$ and $A\left(  Y\right)  $ are the corresponding almost normal subspaces of
$C\left(  X\right)  $ and $C\left(  Y\right)  $.

\begin{prop}
\label{Zx}Suppose that $T:A(X,E)\rightarrow A(Y,F)$ is a surjective linear map
that preserves common zeros. Then there exist a dense subset $Z$ of $\beta Y$
and a homeomorphism $h:Z\rightarrow X$ so that, for all $y\in Z$ and all $f\in
A(X,E)$, there is a net $\left(  y_{\alpha}\right)  $ in $Y$ converging to $y$
such that
\begin{equation}
\left(  Tf\right)  \left(  y_{\alpha}\right)  =\left(  T\mathbf{u}\right)
\left(  y_{\alpha}\right)  \text{ for all }\alpha,\text{ where }u=f\left(
h\left(  y\right)  \right)  . \label{eq1}%
\end{equation}
In particular, $(Tf)^{\beta}(y)=(T\mathbf{u})^{\beta}(y).$
\end{prop}

By (Z) and the compactness of $\beta Y$, the set
\[
Z_{x}=%
{\textstyle\bigcap\limits_{\substack{f\in A\left(  X,E\right)  \\f\left(
x\right)  =0}}}
\overline{Z\left(  Tf\right)  }^{\beta Y}%
\]
is nonempty for all $x\in X$. Before giving the proof of Proposition \ref{Zx},
we first establish a number of lemmas.

\begin{lem}
\label{LFormula}If $y\in Z_{x},$ then for all $f\in A\left(  X,E\right)  ,$
there exists a net $\left(  y_{\alpha}\right)  \subseteq Y$ converging to $y$
such that
\[
\left(  Tf\right)  \left(  y_{\alpha}\right)  =\left(  T\mathbf{u}\right)
\left(  y_{\alpha}\right)  \text{ for all }\alpha,
\]
where $u=f\left(  x\right)  .$ As a result,
\[
\left(  Tf\right)  ^{\beta}\left(  y\right)  =\left(  T\mathbf{u}\right)
^{\beta}\left(  y\right)  .
\]

\end{lem}

\begin{proof}
For any $f\in A\left(  X,E\right)  ,$ $\left(  f-\mathbf{u}\right)  \left(
x\right)  =0,$ where $u=f\left(  x\right)  .$ Hence $y\in Z_{x}\subseteq
\overline{Z\left(  T(f-\mathbf{u})\right)  }^{\beta Y}$. Thus there exists
$\left(  y_{\alpha}\right)  \subseteq Z\left(  T(f-\mathbf{u})\right)  $
converging to $y.$ Then%
\[
\left(  Tf\right)  \left(  y_{\alpha}\right)  =\left(  T\mathbf{u}\right)
\left(  y_{\alpha}\right)  \text{ for all }\alpha.
\]
Taking limits yield
\[
\left(  Tf\right)  ^{\beta}\left(  y\right)  =\left(  T\mathbf{u}\right)
^{\beta}\left(  y\right)  .
\]

\end{proof}

\begin{lem}
\label{Disjoint}If $x_{1}\neq x_{2},$ then $Z_{x_{1}}\cap Z_{x_{2}}%
=\emptyset.$
\end{lem}

\begin{proof}
Suppose that $y_{0}\in Z_{x_{1}}\cap Z_{x_{2}}.$ Take any $f\in A\left(
X,E\right)  $ and let $u=f\left(  x_{1}\right)  .$ Since $y_{0}\in Z_{x_{1}},$
by Lemma \ref{LFormula},%
\[
\left(  Tf\right)  ^{\beta}\left(  y_{0}\right)  =\left(  T\mathbf{u}\right)
^{\beta}\left(  y_{0}\right)  .
\]
There exists $\varphi\in A\left(  X\right)  $ such that $\varphi\left(
x_{1}\right)  =1$ and $\varphi\left(  x_{2}\right)  =0.$ Since $\varphi f\in
A\left(  X,E\right)  $, we may apply Lemma \ref{LFormula} to $\varphi f$ (at
$x_{1}$) to obtain
\[
\left(  T\mathbf{u}\right)  ^{\beta}\left(  y_{0}\right)  =\left(  T\left(
\varphi f\right)  \right)  ^{\beta}\left(  y_{0}\right)  .
\]
However, since $y_{0}\in Z_{x_{2}},$ by Lemma \ref{LFormula} yet again,
$\left(  T\left(  \varphi f\right)  \right)  ^{\beta}\left(  y_{0}\right)
=\left(  T\mathbf{v}\right)  ^{\beta}\left(  y_{0}\right)  ,$ where
$v=\varphi\left(  x_{2}\right)  f\left(  x_{2}\right)  =0.$ Thus
\[
\left(  Tf\right)  ^{\beta}\left(  y_{0}\right)  =\left(  T\mathbf{v}\right)
^{\beta}\left(  y_{0}\right)  =0.
\]
This contradicts the surjectivity of $T,$ since $\mathbf{w\in}A\left(
Y,F\right)  $ for any $w\in F\smallsetminus\left\{  0\right\}  .$
\end{proof}

\begin{lem}
\label{Singleton}$\left\vert Z_{x}\right\vert =1$ for all $x\in X.$
\end{lem}

\begin{proof}
Suppose, to the contrary, that there are distinct $y_{0},y_{1}\in Z_{x}.$ Let
$U_{0},U_{1}$ be open neighborhoods of $y_{0}$ and $y_{1}$ in $\beta
Y\ $respectively$\ $whose closures are disjoint. By almost normality of
$A\left(  Y\right)  ,$ there exists $\varphi\in A\left(  Y\right)  $ such
that
\[
\varphi=0\text{ on }U_{0}\cap Y\text{ and }\varphi=1\text{ on \thinspace}%
U_{1}\cap Y.
\]
Let $v\in F\smallsetminus\left\{  0\right\}  $ and $g=\varphi\cdot
\mathbf{v}\in A\left(  Y,F\right)  .$ From the fact that $T$ is onto$,$ there
exists $f\in A\left(  X,E\right)  $ such that $Tf=g.$ By Lemma \ref{LFormula},
for $i=0,1,$ there exist $\left(  y_{\alpha}^{i}\right)  $ in $Y$ converging
to $y_{i}$ such that%
\[
\left(  T\mathbf{u}\right)  \left(  y_{\alpha}^{i}\right)  =\left(  Tf\right)
\left(  y_{\alpha}^{i}\right)  =\varphi\left(  y_{\alpha}^{i}\right)  v,\text{
where }u=f\left(  x\right)  .
\]
Since $\varphi\left(  y_{\alpha}^{0}\right)  =0$ for sufficiently large
$\alpha,$ $\left(  T\mathbf{u}\right)  \left(  y_{\alpha}^{0}\right)  =0$ for
such $\alpha$'s. It follows that $Z\left(  T\mathbf{u}\right)  \neq\emptyset.$
By (Z), $Z\left(  \mathbf{u}\right)  \neq\emptyset$ and hence $f\left(
x\right)  =u=0.$ On the other hand, $\lim\limits_{\alpha}\varphi\left(
y_{\alpha}^{1}\right)  v=v\neq0.$ But
\[
\varphi\left(  y_{\alpha}^{1}\right)  v=\left(  Tf\right)  \left(  y_{\alpha
}^{1}\right)  =\left(  T\mathbf{u}\right)  \left(  y_{\alpha}^{1}\right)
=\left(  T\mathbf{0}\right)  \left(  y_{\alpha}^{1}\right)  =0,
\]
a contradiction.
\end{proof}

\begin{proof}
[Proof of Proposition \ref{Zx}]Let $Z=\bigcup\limits_{x}Z_{x}.$ Then
$Z\subseteq\beta Y.$ Define a mapping $h:Z\rightarrow X$ by sending the unique
element in $Z_{x}$ to $x.$ The mapping $h$ is well-defined by Lemma
\ref{Disjoint}. Clearly $h$ is a bijection and for any $y\in Z$ and $f\in
A\left(  X,E\right)  ,$ Lemma \ref{LFormula} yields a net $\left(  y_{\alpha
}\right)  $ in $Y$ converging to $y$ so that (\ref{eq1}) is satisfied. It
remains to show that $h$ and $h^{-1}$ are continuous, and that $Z$ is dense in
$\beta Y$.

Suppose that $h$ is not continuous. Using the compactness of $\beta X,$ there
exists $y_{0}\in Z,$ $x^{\prime}\in\beta X$ and a net $\left(  y_{\alpha
}\right)  $ in~$Z$ converging to $y_{0}$ such that
\[
h\left(  y_{\alpha}\right)  =x_{\alpha}\rightarrow x^{\prime}\neq
x_{0}=h\left(  y_{0}\right)  .
\]
Let $U$ be an open neighborhood of $x^{\prime}$ in $\beta X$ such that
$\overline{U}^{\beta X}$ does not contain $x_{0}.$ Choose $\varphi\in A\left(
X\right)  $ such that $\varphi=0$ in $U\cap X$ and $\varphi\left(
x_{0}\right)  =1.$ For all $f\in A\left(  X,E\right)  .$ We have by Lemma
\ref{LFormula} that
\[
\left(  T\left(  \varphi f\right)  \right)  \left(  y_{\alpha}\right)
=\left(  T\mathbf{u}_{\alpha}\right)  \left(  y_{\alpha}\right)  \text{ for
all }\alpha,
\]
where $u_{\alpha}=\varphi\left(  x_{\alpha}\right)  f\left(  x_{\alpha
}\right)  .$ Since $u_{\alpha}=\varphi\left(  x_{\alpha}\right)  f\left(
x_{\alpha}\right)  =0$ for sufficiently large $\alpha,$ $\left(  T\left(
\varphi f\right)  \right)  ^{\beta}\left(  y_{0}\right)  =0.$ On the other
hand, $\left(  T\left(  \varphi f\right)  \right)  ^{\beta}\left(
y_{0}\right)  =\left(  T\mathbf{v}\right)  ^{\beta}\left(  y_{0}\right)  ,$
where $v=\varphi\left(  x_{0}\right)  f\left(  x_{0}\right)  =f\left(
x_{0}\right)  .$ Thus $\left(  Tf\right)  ^{\beta}\left(  y_{0}\right)  =0$
for all $f\in A\left(  X,E\right)  ,$ contradicting the surjectivity of $T.$

Suppose that $h^{-1}$ is not continuous. Using the compactness of $\beta Y,$
there exists $x_{0}\in X,$ $y^{\prime}\in\beta Y$ and a net $\left(
x_{\alpha}\right)  $ in~$X$ converging to $x_{0}$ such that $h^{-1}\left(
x_{\alpha}\right)  =y_{\alpha}\rightarrow y^{\prime}\neq y_{0}=h^{-1}\left(
x_{0}\right)  .$ Let\ $U$ and $V$ be open neighborhoods of $y_{0}$ and
$y^{\prime}$ in $\beta Y$ respectively whose closures are disjoint. Choose
$\varphi\in A\left(  Y\right)  $ such that $\varphi=1$ in $U\cap Y$ and
$\varphi=0$ on $V\cap Y.$ Let $v\in F\smallsetminus\left\{  0\right\}  \ $and
set $g=\varphi\cdot\mathbf{v}\in A\left(  Y,F\right)  .$ Since $T$ is
surjective, there is an $f\in A\left(  X,E\right)  $ such that $Tf=g.$ By
Lemma \ref{LFormula}, for each $\alpha,$ there exists a net $\left(
y_{\alpha}^{\gamma}\right)  _{\gamma}$ in $Y$ converging to $y_{\alpha}$ such
that%
\begin{align*}
\left(  T\mathbf{u}_{\alpha}\right)  \left(  y_{\alpha}^{\gamma}\right)   &
=\left(  Tf\right)  \left(  y_{\alpha}^{\gamma}\right)  \text{ for all }%
\gamma,\text{ where }u_{\alpha}=f\left(  x_{\alpha}\right)  ,\\
&  =g\left(  y_{\alpha}^{\gamma}\right)  =\varphi\left(  y_{\alpha}^{\gamma
}\right)  v\text{ for all }\gamma.
\end{align*}
For sufficiently large $\alpha,$ there exists $\gamma$ where $\varphi\left(
y_{\alpha}^{\gamma}\right)  =0.$ Hence $Z\left(  T\mathbf{u}_{\alpha}\right)
\neq\emptyset$ for large enough $\alpha.$ By (Z), $Z\left(  \mathbf{u}%
_{\alpha}\right)  \neq\emptyset$ and consequently, $f\left(  x_{\alpha
}\right)  =0.$ Thus $f\left(  x_{0}\right)  =0.$ Therefore, by Lemma
\ref{LFormula},
\begin{align*}
0 &  =\left(  T\mathbf{0}\right)  ^{\beta}\left(  y_{0}\right)  =\left(
Tf\right)  ^{\beta}\left(  y_{0}\right)  \\
&  =\left(  \varphi\cdot\mathbf{v}\right)  ^{\beta}\left(  y_{0}\right)  \\
&  =\varphi^{\beta}\left(  y_{0}\right)  v\\
&  =v\neq0,
\end{align*}
a contradiction.

Finally, we show that $Z$ is dense in $\beta Y.$ If this is not true, then
there exists $y_{0}\in\beta Y\smallsetminus\overline{Z}^{\beta Y}.$ Let $U$
and $V$ be open neighborhoods of $\overline{Z}^{\beta Y}$ and $y_{0}$ in
$\beta Y$ respectively whose closures are disjoint. There exists $\varphi\in
A\left(  Y\right)  $ so that $\varphi=0$ on $U\cap Y$ and $\varphi=1$ on
$V\cap Y.$ Take $v\in F\smallsetminus\left\{  0\right\}  $ and set
$g=\varphi\cdot\mathbf{v}\in A\left(  Y,F\right)  .$ By the surjectivity of
$T,$ there exists $f\in A\left(  X,E\right)  $ such that $Tf=g.$ Let $x\in X$
and $y\in Z_{x}\subseteq U.$ By Lemma \ref{LFormula}, there exists $\left(
y_{\alpha}\right)  \subseteq Y$ converging to $y$ such that
\begin{align*}
\left(  T\mathbf{u}\right)  \left(  y_{\alpha}\right)   &  =\left(  Tf\right)
\left(  y_{\alpha}\right)  \text{ for all }\alpha\text{, where }u=f\left(
x\right) \\
&  =\varphi\left(  y_{\alpha}\right)  v.
\end{align*}
For sufficiently large $\alpha,$ $\varphi\left(  y_{\alpha}\right)  =0.$ Thus
$Z\left(  T\mathbf{u}\right)  \neq\emptyset$ and by $\left(  \text{Z}\right)
,$ $Z\left(  \mathbf{u}\right)  \neq\emptyset,$ which means that $u=0.$ Since
$x$ is arbitrary, $f=0.$ Thus $g=0,$ which is absurd.
\end{proof}

The next theorem, which is the main result of the section, gives a preliminary
description of vector space isomorphisms preserving common zeros. Applications
of this theorem in certain cases yield sharp characterizations of said
mappings. These applications will be our concern in the subsequent sections.

\begin{thm}
\label{thm2}Suppose that $T:A(X,E)\rightarrow A(Y,F)$ is a vector space
isomorphism that preserves common zeros. Then there are dense subsets $Z$ of
$\beta Y$ and $W$ of $\beta X$ and homeomorphisms $h:Z\rightarrow X$,
$k:W\rightarrow Y$ so that $h\cup k^{-1}:Z\cup Y\rightarrow X\cup W$ is a
homeomorphism. Moreover, for all $f\in A(X,E),\ y\in Z,$%
\[
(Tf)^{\beta}(y)=(T\mathbf{u})^{\beta}(y),\ u=f(h(y)),
\]
and for all $\ g\in A(Y,F),\ x\in W,$
\[
(T^{-1}g)^{\beta}(x)=(T^{-1}\mathbf{v})^{\beta}(x),\ v=g(k(x)).
\]

\end{thm}

\begin{proof}
Applying Proposition \ref{Zx} to both $T$ and $T^{-1}$, we find dense subsets
$Z$ of $\beta Y,$ $W$ of $\beta X$ and homeomorphisms $h:Z\rightarrow X$,
$k:W\rightarrow Y$ such that
\begin{equation}
(Tf)^{\beta}(y)=(T\mathbf{u})^{\beta}(y),\ u=f(h(y)) \label{FormulaT}%
\end{equation}
whenever$\ f\in A(X,E),$ $y\in Z$ and%
\begin{equation}
(T^{-1}g)^{\beta}(x)=(T^{-1}\mathbf{v})^{\beta}(x),\ v=g(k(x))
\label{FormulaTinv}%
\end{equation}
whenever $g\in A\left(  Y,F\right)  ,$ $x\in W.$

In the notation of Lemmas \ref{Extension} and \ref{WellDefined}, $\tilde
{h}_{|Z\cup Y}=\tilde{h}_{|Z}\cup\tilde{h}_{|Y}=h\cup k^{-1}$ is a
well-defined map from $Z\cup Y$ to $X\cup W.$ By symmetry, $h^{-1}\cup k$ is
also a well-defined continuous map. Therefore, $h\cup k^{-1}$ is a
homeomorphism. The proof of Theorem \ref{thm2} is complete.
\end{proof}

\begin{lem}
\label{Extension}There is a continuous extension $\tilde{h}:\beta
Y\rightarrow\beta X$ of $h:Z\rightarrow X.$
\end{lem}

\begin{proof}
If the lemma fails, then there exist $\left(  y_{\alpha}^{1}\right)  ,\left(
y_{\alpha}^{2}\right)  $ in $Z$ converging to $y_{0}\in\beta Y$ such that
$h\left(  y_{\alpha}^{i}\right)  \rightarrow x_{i}\in\beta X,$ $i=1,2,$ with
$x_{1}\neq x_{2}.$ Let $U$ and $V$ be open neighborhoods of $x_{1}$ and
$x_{2}$ in $\beta X$ respectively with disjoint closures. There exists
$\varphi\in A\left(  X\right)  $ such that $\varphi=0$ on $U\cap X$ and
$\varphi=1$ on $V\cap X.\,$Let $v\in F\smallsetminus\left\{  0\right\}  $ and
pick $f\in A\left(  X,E\right)  $ such that $Tf=\mathbf{v}.$ Applying
(\ref{FormulaT}), we find that for all $\alpha,$%

\[
v =\left(  Tf\right)  ^{\beta}\left(  y_{\alpha}^{2}\right)  =\left(
T\mathbf{w}_{\alpha}^{2}\right)  ^{\beta}\left(  y_{\alpha}^{2}\right)
,\text{ where }w_{\alpha}^{2}=f\left(  h\left(  y_{\alpha}^{2}\right)
\right)
\]
and%
\[
\left(  T\left(  \varphi f\right)  \right)  ^{\beta}\left(  y_{\alpha}%
^{i}\right)  =\left(  T\mathbf{u}_{\alpha}^{i}\right)  ^{\beta}\left(
y_{\alpha}^{i}\right)  \text{, where }u_{\alpha}^{i}=\left(  \varphi f\right)
\left(  h\left(  y_{\alpha}^{i}\right)  \right)  ,\text{ }i=1,2.
\]
For sufficiently large $\alpha,$ $\varphi\left(  h\left(  y_{\alpha}%
^{1}\right)  \right)  =0$ and $\varphi\left(  h\left(  y_{\alpha}^{2}\right)
\right)  =1\ $and thus%
\[
u_{\alpha}^{i}=\left\{
\begin{array}
[c]{ccc}%
0 & \text{if} & i=1,\\
w_{\alpha}^{2} & \text{if} & i=2.
\end{array}
\right.
\]
It follows that for such $\alpha,$
\begin{align*}
\left(  T\left(  \varphi f\right)  \right)  ^{\beta}\left(  y_{\alpha}%
^{i}\right)   &  =\left\{
\begin{array}
[c]{ccc}%
0 & \text{if} & i=1,\\
\left(  Tf\right) ^{\beta}\left(  y_{\alpha}^{2}\right)  & \text{if} & i=2,
\end{array}
\right. \\
&  =\left\{
\begin{array}
[c]{ccc}%
0 & \text{if} & i=1,\\
v & \text{if} & i=2.
\end{array}
\right.
\end{align*}
Upon taking limits, we have $v=\left(  T\left(  \varphi f\right)  \right)
^{\beta}\left(  y_{0}\right)  =0,$ a contradiction.
\end{proof}

\begin{lem}
\label{WellDefined}$\tilde{h}_{|Y}=k^{-1}.$
\end{lem}

\begin{proof}
Suppose that there exists $y_{0}\in Y$ such that $x^{\prime}=\tilde{h}\left(
y_{0}\right)  $ is different from $x_{0}=k^{-1}\left(  y_{0}\right)  .$ Let
$\left(  y_{\alpha}\right)  $ be a net in $Z$ converging to $y_{0}$ and let
$x_{\alpha}=\tilde{h}\left(  y_{\alpha}\right)  =h\left(  y_{\alpha}\right)
.$ Then $x_{\alpha}\rightarrow x^{\prime}.$ Let $U$ and $V$ be respective open
neighborhoods of $x^{\prime}$ and $x_{0}$ in $\beta X$ with disjoint closures.
Choose $\varphi\in A\left(  X\right)  $ such that $\varphi=0$ on $U\cap X$ and
$\varphi=1$ on $V\cap X$ and set $f=\varphi\cdot\mathbf{u}$ for some fixed
$u\in E\smallsetminus\left\{  0\right\}  .$ For all $\alpha,$
\begin{align*}
\left(  Tf\right)  ^{\beta}\left(  y_{\alpha}\right)   &  =(T\mathbf{u}%
_{\alpha})^{\beta}\left(  y_{\alpha}\right)  ,\text{ where }u_{\alpha
}=f\left(  x_{\alpha}\right) \\
&  =0\text{ \ when }\alpha\text{ is sufficiently large.}%
\end{align*}
Thus $0=\left(  Tf\right)  ^{\beta}\left(  y_{0}\right)  =\left(  Tf\right)
\left(  y_{0}\right)  ,$ since $y_{0}\in Y.$ It follows by (\ref{FormulaTinv})
that%
\[
f^{\beta}\left(  x_{0}\right)  =\left(  T^{-1}\left(  Tf\right)  \right)
^{\beta}\left(  x_{0}\right)  =\left(  T^{-1}\mathbf{0}\right)  ^{\beta
}\left(  x_{0}\right)  =0,
\]
contradicting the fact that $f(x)=u\neq0$ for $x \in V\cap X.$
\end{proof}

\section{Spaces of continuous functions}

In this section, we apply the results of the previous section to the case
where $A(X,E)=C(X,E)$ and $A(Y,F)=C(Y,F)$. Obviously, these spaces are almost
normally multiplicative. Since the one-point compactification ${\mathbb{R}%
}_{\infty}$ of ${\mathbb{R}}$ is compact, every $f\in C(X)$ has a unique
extension to a continuous function $f^{\ast}:\beta X\rightarrow{\mathbb{R}%
}_{\infty}$. Recall that the \emph{Hewitt realcompactification} $\upsilon X$
of $X$ \cite{GJ} is the set
\[
\{x\in\beta X:f^{\ast}(x)\in{\mathbb{R}}\text{ for all }f\in C(X)\}.
\]
$X$ is said to be \emph{realcompact} if $X=\upsilon X$.

\begin{prop}
\label{prop3} Let $A(X,E)=C(X,E)$ and $A(Y,F)=C(Y,F)$ in Proposition \ref{Zx}.
Then the set $Z$ obtained in that Proposition is a subset of $\upsilon Y$.
\end{prop}

\begin{proof}
Suppose, to the contrary, that there exists $y_{0}\in Z\smallsetminus\upsilon
Y.$ Then there exists $\psi\in C\left(  Y\right)  $ such that $\psi\left(
y\right)  \geq1$ for all $y\in Y$ and $\psi^{\ast}\left(  y_{0}\right)
=\infty.$ Clearly $\psi^{\ast}\left(  y\right)  \neq0$ for all $y\in\beta Y.$
Define $\varphi:X\rightarrow\mathbb{R}$ by
\[
\varphi\left(  x\right)  =\frac{1}{\psi^{\ast}\left(  h^{-1}\left(  x\right)
\right)  } \text{\quad(taking }\frac{1}{\infty}=0\text{)}.
\]
Then $\varphi\in C\left(  X\right)  .$ Let $v\in F\smallsetminus\left\{
0\right\}  $ and choose $f\in C\left(  X,E\right)  $ such that $Tf=\psi
\cdot\mathbf{v}.$\newline

\noindent\textbf{Claim 1.} If $y\in Z$ and $\psi^{\ast}\left(  y\right)
\in\mathbb{R},$ then $\left(  T\left(  \varphi f\right)  \right)  ^{\beta
}\left(  y\right)  =v.$

\noindent\emph{Proof of Claim 1.} Suppose that $y\in Z$ and $\psi^{\ast
}\left(  y\right)  \in\mathbb{R}$. Let $c=\varphi\left(  h\left(  y\right)
\right)  $ and $u=f\left(  h\left(  y\right)  \right)  .$ By Proposition
\ref{Zx},
\begin{align*}
\left(  T\mathbf{u}\right)  ^{\beta}\left(  y\right)   &  =\left(  Tf\right)
^{\beta}\left(  y\right) \\
&  =\left(  \psi\cdot\mathbf{v}\right)  ^{\beta}\left(  y\right) \\
&  =\psi^{\ast}\left(  y\right)  v.
\end{align*}
Since $c\psi^{\ast}\left(  y\right)  =\varphi\left(  h\left(  y\right)
\right)  \psi^{\ast}\left(  y\right)  =1,$ we have
\begin{align*}
v  &  =c\psi^{\ast}\left(  y\right)  v\\
&  =c\left(  T\mathbf{u}\right)  ^{\beta}\left(  y\right) \\
&  =\left(  T\left(  c\cdot\mathbf{u}\right)  \right)  ^{\beta}\left(
y\right)  ,
\end{align*}
where the last equality holds because $c\in\mathbb{R}$ and $\left(
T\mathbf{u}\right)  ^{\beta}\left(  y\right)  \in F.$ Finally,
\begin{align*}
\left(  T\left(  \varphi f\right)  \right)  ^{\beta}\left(  y\right)   &
=\left(  T\mathbf{w}\right)  ^{\beta}\left(  y\right)  ,\text{ where
}w=\varphi\left(  h\left(  y\right)  \right)  f\left(  h\left(  y\right)
\right)  =cu,\\
&  =\left(  T\left(  c\cdot\mathbf{u}\right)  \right)  ^{\beta}\left(
y\right)  =v.
\end{align*}

\noindent\textbf{Claim 2.} There exists $\left(  y_{\alpha}\right)  $ in $Z$
converging to $y_{0}$ such that $\psi^{\ast}\left(  y_{\alpha}\right)
\in\mathbb{R}$ for all $\alpha.$

\noindent\emph{Proof of Claim 2.} It is enough to show that for all open
neighborhoods $V$ of $y_{0}$ in $\beta Y,$ there exists $y\in V\cap Z$ with
$\psi^{\ast}\left(  y\right)  \in\mathbb{R}$. Let $V$ be an open neighborhood
of $y_{0}$ in $\beta Y$ and choose $y_{1}\in V\cap Y.$ Since $\psi^{\ast
}\left(  y_{1}\right)  =\psi\left(  y_{1}\right)  \in\mathbb{R},$ there is an
open neighborhood $V_{1}$ of $y_{1}$ in $\beta Y$ such that $\psi^{\ast
}\left(  y\right)  \in\mathbb{R}$ for all $y\in V_{1}.$ Now $V\cap V_{1}$ is
an open set in $\beta Y$ containing $y_{1}$ and thus is nonempty. Since $Z$ is
dense in $\beta Y,$ $Z\cap V\cap V_{1}\neq\emptyset.$ If $y\in Z\cap V\cap
V_{1},$ then $y\in Z\cap V$ and $\psi^{\ast}\left(  y\right)  \in\mathbb{R}%
.$\newline

Returning to the proof of the proposition, let $\left(  y_{\alpha}\right)  $
be chosen using Claim 2. It follows from Claim 1 that
\[
v=\left(  T\left(  \varphi f\right)  \right)  ^{\beta}\left(  y_{\alpha
}\right)  \text{ for all }\alpha.
\]
Thus, $v=\left(  T\left(  \varphi f\right)  \right)  ^{\beta}\left(
y_{0}\right)  .$ On the other hand, by Proposition \ref{Zx},
\begin{align*}
\left(  T\left(  \varphi f\right)  \right)  ^{\beta}\left(  y_{0}\right)   &
=\left(  T\mathbf{u}_{0}\right)  ^{\beta}\left(  y_{0}\right)  ,\text{ where
}u_{0}=\varphi\left(  x_{0}\right)  f\left(  x_{0}\right)  \text{ and }%
x_{0}=h\left(  y_{0}\right)  ,\\
&  =\left(  T\mathbf{0}\right)  ^{\beta}\left(  y_{0}\right)  ,\text{ as
}\varphi\left(  x_{0}\right)  =\frac{1}{\psi^{\ast}\left(  y_{0}\right)
}=0,\\
&  =0.
\end{align*}
Hence $v=0,$ contrary to the choice of $v.$
\end{proof}

\begin{thm}
\label{thm4} Let $X$ and $Y$ be realcompact spaces and let $E$ and $F$ be
Hausdorff topological vector spaces. Suppose that $T:C(X,E)\rightarrow C(Y,F)$
is a vector space isomorphism that preserves common zeros. Then there are a
homeomorphism $h:Y\rightarrow X$ and, for each $y\in Y$, a vector space
isomorphism $S_{y}:E\rightarrow F$ so that
\begin{equation}
Tf(y)=S_{y}(f(h(y)))\text{ for all $f\in C(X,E)$ and all $y\in Y$.} \label{Tf}%
\end{equation}
Conversely, if a vector space isomorphism $T:C(X,E)\rightarrow C(Y,F)$ has the
form (\ref{Tf}), then $T$ preserves common zeros.
\end{thm}

\begin{proof}
By Theorem \ref{thm2}, there exist homeomorphisms $h:Z\rightarrow X$,
$k:W\rightarrow Y$ so that $h\cup k^{-1}:Z\cup Y\rightarrow X\cup W$ is a
homeomorphism. According to Proposition \ref{prop3}, $Z\subseteq\upsilon Y$.
Since $Y$ is realcompact$,$ we deduce that $Z\subseteq Y$ and hence $h$ is a
restriction of $k^{-1}.$ Similarly, $W\subseteq X.$ Then $k^{-1}\left(
Y\right)  =W\subseteq X=h\left(  Z\right)  $ . So we must have $X=W$ and
$Y=Z.$ Therefore, $h$ is a homeomorphism from $Y$ onto $X.$

For each $y\in Y,$ define a linear operator $S_{y}:E\rightarrow F$ by%
\[
S_{y}\left(  u\right)  =\left(  T\mathbf{u}\right)  \left(  y\right)  .
\]
If $u\in\ker S_{y},$ then $y\in Z\left(  T\mathbf{u}\right)  .$ Thus $Z\left(
T\mathbf{u}\right)  \neq\emptyset$ and hence $Z\left(  \mathbf{u}\right)
\neq\emptyset.$ Consequently, $u=0.$ The surjectivity of $S_{y}$ follows
easily from that of $T.$ Therefore, each $S_{y}$ is a linear isomorphism.
Since $Z=Y,$ it follows from Theorem \ref{thm2} that%
\[
\left(  Tf\right)  \left(  y\right)  =S_{y}(f\left(  h(y)\right)  )\text{ for
all }y\in Y.
\]
The converse is clear.
\end{proof}

The following corollary of Theorem \ref{thm4} contains the result conjectured
in \cite{EO}.  The original conjecture was solved by Chen, Chen and Wong \cite{CCW} and independently by Ercan and \"{O}nal \cite{EO2}.

\begin{thm}
\label{lattice} Let $X$ and $Y$ be realcompact spaces and let $E$ and $F$ be
Hausdorff topological vector lattices. Suppose that $T:C(X,E)\rightarrow
C(Y,F)$ is a vector lattice isomorphism so that
\[
Z(f)\neq\emptyset\iff Z(Tf)\neq\emptyset.
\]
Then there is a homeomorphism $h:Y\rightarrow X$ and, for each $y\in Y$, a
vector lattice isomorphism $S_{y}:E\rightarrow F$ so that
\[
Tf(y)=S_{y}(f\left(  h(y)\right)  )\text{ for all $f\in C(X,E)$ and all $y\in
Y$.}%
\]

\end{thm}

\begin{proof}
We first show that $T$ preserves common zeros. Indeed, if $f_{1},\cdots
,f_{k}\in C\left(  X,E\right)  ,$ set $f=\left\vert f_{1}\right\vert
\vee\cdots\vee\left\vert f_{k}\right\vert .$ It is clear that $Z\left(
f\right)  =\cap_{i=1}^{k}Z(f_{i}).$ Since $T$ is a lattice isomorphism,
$Tf=\left\vert Tf_{1}\right\vert \vee\cdots\vee\left\vert Tf_{k}\right\vert .$
Thus,%
\begin{align*}
\emptyset &  \neq\cap_{i=1}^{k}Z(f_{i})=Z\left(  f\right) \\
&  \Leftrightarrow Z\left(  Tf\right)  \neq\emptyset,\text{ by hypothesis,}\\
&  \Leftrightarrow\emptyset\neq Z\left(  \left\vert Tf_{1}\right\vert
\vee\cdots\vee\left\vert Tf_{k}\right\vert \right)  =\cap_{i=1}^{k}Z(Tf_{i}).
\end{align*}
By Theorem \ref{thm4} , we obtain a homeomorphism $h : Y \to X$ and vector
space isomorphisms $S_{y},$ $y\in Y$, satisfying (\ref{Tf}). For any $y\in Y$
and $u\in E,$ the equation $\left(  T\left\vert \mathbf{u}\right\vert \right)
\left(  y\right)  =\left\vert T\mathbf{u}\right\vert \left(  y\right)  $ shows
that $S_{y}\left(  \left\vert u\right\vert \right)  =\left\vert S_{y}\left(
u\right)  \right\vert .$ Thus $S_{y}$ is a vector lattice isomorphism.
\end{proof}

Theorem \ref{thm4} holds equally if the scalar field is assumed to be
${\mathbb{C}}$.

\begin{thm}
\label{C*} Let $X$ and $Y$ be realcompact spaces and let $E$ and $F$ be
$C^{\ast}$-algebras. Suppose that $T:C(X,E)\rightarrow C(Y,F)$ is a $\ast
$-algebra isomorphism so that
\[
Z(f)\neq\emptyset\iff Z(Tf)\neq\emptyset.
\]
Then there is a homeomorphism $h:Y\rightarrow X$ and, for each $y\in Y$, a
$C^{\ast}$-algebra isomorphism $S_{y}:E\rightarrow F$ so that
\[
Tf(y)=S_{y}(f\circ h(y))\text{ for all $f\in C(X,E)$ and all $y\in Y$.}%
\]

\end{thm}

\begin{proof}
If $f_{1},\cdots,f_{k}\in C\left(  X,E\right)  $,\ let $f=\sum_{i=1}^{k}%
f_{i}f_{i}^{\ast}.$ It is clear that $Z\left(  f\right)  =\cap_{i=1}%
^{k}Z(f_{i}).$ Since $T$ is a $\ast$-isomorphism, $Tf=\sum_{i=1}^{k}%
Tf_{i}\left(  Tf_{i}\right)  ^{\ast}.$ Thus,%
\begin{align*}
\emptyset &  \neq\cap_{i=1}^{k}Z(f_{i})=Z\left(  f\right) \\
&  \Leftrightarrow Z\left(  Tf\right)  \neq\emptyset,\text{ by hypothesis,}\\
&  \Leftrightarrow\emptyset\neq Z(%
{\textstyle\sum_{i=1}^{k}}
Tf_{i}\left(  Tf_{i}\right)  ^{\ast})=\cap_{i=1}^{k}Z(Tf_{i}).
\end{align*}
Therefore $T$ preserves common zeros. The rest of the proof follows along the
lines of the proof of Theorem \ref{lattice}.
\end{proof}

\section{Spaces of differentiable functions}

In this section, we fix $p,q\in{\mathbb{N}}$ and let $X$ and $Y$ be open
subsets of ${\mathbb{R}}^{p}$ and ${\mathbb{R}}^{q}$ respectively. The results
of \S \ref{sec1} are applied to spaces of differentiable functions
$C^{m}(X,E)$ and $C^{n}(Y,F)$, where $m,n\in{\mathbb{N}}\cup\{0, \infty\}$.
Here $C^{m}\left(  X,E\right)  $ denotes the space of functions from $X$ into
$E$ having continuous partial derivatives of all order $<m+1\left(  \infty
+1=\infty\right)  .$ Note that the spaces $C^{m}(X)$ and $C^{n}(Y)$ are almost
normal and thus $C^{m}(X,E)$ and $C^{n}(Y,F)$ are almost normally multiplicative.

\begin{lem}
\label{lem1}Let $X$ and $Y$ be open subsets of ${\mathbb{R}}^{p}$ and
${\mathbb{R}}^{q}$ respectively, $p,q\in{\mathbb{N}}$. Suppose that $Z$ is a
dense subspace of $\beta Y$ that is homeomorphic to $X$. Then $Z\subseteq Y$.
\end{lem}

\begin{proof}
Suppose that $h:Z\rightarrow X$ is a homeomorphism. If $Z$ $\nsubseteqq Y,$
then there exists $y_{0}\in Z\smallsetminus Y$ such that $h\left(
y_{0}\right)  =x_{0}\in X.$ For each $n$, let $U_{n}$ denote the open set
$h^{-1}\left(  B\left(  x_{0},\frac{1}{n}\right)  \right)  $ in $Z.$ Then
$U_{n}=V_{n}\cap Z$ for some open subset $V_{n}$ of $\beta Y.$ Since $Y$ is
locally compact, it is open in $\beta Y.$ Thus, $V_{n}\cap Y$ is a nonempty
open set in $\beta Y$. Therefore, $V_{n}\cap Y\cap Z\neq\emptyset.$ For each
$n,$ pick $y_{n}\in V_{n}\cap Y\cap Z.$ Since $h\left(  y_{n}\right)
\rightarrow x_{0}$ and $h^{-1}$ is continuous, $\left(  y_{n}\right)  $
converges to $y_{0}.$ We may assume without loss of generality that $\left(
y_{n}\right)  $ is pairwise distinct and has no accumulation point in $Y.$
There exists $g\in C\left(  Y\right)  $ such that $0\leq g\leq1,$ $g\left(
y_{2n-1}\right)  =0$ and $g\left(  y_{2n}\right)  =1$ for all $n.$ Consider
the continuous extension $g^{\#}:\beta Y\rightarrow\left[  0,1\right]  $ of
$g.$ By the continuity of $g^{\#},$
\[
1 =\lim_{n\rightarrow\infty}g\left(  y_{2n}\right)
=g^{\#}\left(  y_{0}\right)  =\lim_{n\rightarrow\infty}g\left(  y_{2n-1}%
\right)  =0,
\]
a contradiction.
\end{proof}

\begin{lem}
\noindent\label{Cancellation}Let $\varphi$ be a real-valued function on $Y$
and $f$ be a function in $C\left(  Y,F\right)  $ that is never zero$.$ Assume
that $\lim_{y\rightarrow y_{0}}\varphi\left(  y\right)  f\left(  y\right)  =v$
exists. Then $\lim_{y\rightarrow y_{0}}\varphi\left(  y\right)  =a$ exists and
$a\cdot f\left(  y_{0}\right)  =v$.
\end{lem}

\begin{proof}
We first show that $\varphi$ is bounded in a neighborhood of $y_{0}$. Suppose
otherwise. Then there is a sequence $\left(  y_{n}\right)  $ converging
nontrivially to $y_{0}$ such that $\left\vert \varphi\left(  y_{n}\right)
\right\vert \geq n$ for all $n.$ Since $\lim_{y\rightarrow y_{0}}%
\varphi\left(  y\right)  f\left(  y\right)  =v$ exists, for any circled
neighborhood $U$ of $0,$ $\varphi\left(  y_{n}\right)  f\left(  y_{n}\right)
\in v+U$ for sufficiently large $n.$ Thus
\[
f\left(  y_{n}\right)  -\frac{v}{\varphi\left(  y_{n}\right)  }\in\frac
{1}{\varphi\left(  y_{n}\right)  } U\subseteq U
\]
for sufficiently large $n.$ Hence $\lim_{n\rightarrow\infty}\left(  f\left(
y_{n}\right)  -\frac{v}{\varphi\left(  y_{n}\right)  }\right)  =0.\,\ $By the
continuity of $f,\ f\left(  y_{0}\right)  =0,$ a contradiction.

Since $\varphi$ is bounded in a neighborhood of $y_{0},$ every sequence
$\left(  y_{n}\right)  $ that converges nontrivially to $y_{0}$ has a
subsequence $\left(  y_{n_{k}}\right)  $ such that $\left(  \varphi\left(
y_{n_{k}}\right)  \right)  $ converges. Suppose that $\left(  y_{n}\right)  $
and $\left(  z_{n}\right)  $ are sequences converging nontrivially to $y_{0}$
such that
\[
\lim_{n\rightarrow\infty}\varphi\left(  y_{n}\right)  =L_{1}\text{ and }%
\lim_{n\rightarrow\infty}\varphi\left(  z_{n}\right)  =L_{2}.\text{ }%
\]
Then
\[
L_{1}f\left(  y_{0}\right)  =\lim_{n\rightarrow\infty}\varphi\left(
y_{n}\right)  f\left(  y_{n}\right)  =\lim_{n\rightarrow\infty}\varphi\left(
z_{n}\right)  f\left(  z_{n}\right)  =L_{2}f\left(  y_{0}\right)  .
\]
Since $f\left(  y_{0}\right)  \neq0,$ $L_{1}=L_{2}.$ Hence $\lim_{y\rightarrow
y_{0}}\varphi\left(  y\right)  =a$ exists. Clearly $a\cdot f\left(
y_{0}\right)  =v.$
\end{proof}

Let $r\in\mathbb{N}$. A \emph{multi-index} $\lambda$ is an $r$-tuple $\left(
\lambda_{1},...,\lambda_{r}\right)  $ with entries in $\mathbb{N}\cup\left\{
0\right\}  ,$ which will also be regarded as a vector in $\mathbb{R}^{r}.$ The
\emph{order} of $\lambda$ is $\left\vert \lambda\right\vert =\lambda
_{1}+\cdots+\lambda_{r}$. If $f$ is a function of $r$ variables, we denote by
$\partial^{\lambda}f$ the partial derivative $\left(  \partial^{1}\right)
^{\lambda_{1}}\cdots\left(  \partial^{r}\right)  ^{\lambda_{r}}f.$

\begin{lem}
\label{ProductRule}Let $n\in\mathbb{N}\cup\left\{  0\right\}  .$ If $\varphi$
is real-valued function on $Y$ and $\varphi f\in C^{n}\left(  Y,F\right)
\ $for some $f\in C^{n}\left(  Y,F\right)  $ that is never zero$,$ then
$\varphi\in C^{n}\left(  Y\right)  .$
\end{lem}

\begin{proof}
The case $n=0$ follows easily from Lemma \ref{Cancellation}. We prove the
remaining cases by induction. Assume that $n=1.\,$Let $\lambda$ be a
multi-index with $\left\vert \lambda\right\vert =1.$ For all $y_{0}\in Y$ and
all $t\neq0,$ set $y_{t}=y_{0}+t\lambda.$ Then%
\begin{align*}
&  \lim_{t\rightarrow0}\frac{\varphi\left(  y_{t}\right)  -\varphi\left(
y_{0}\right)  }{t}f\left(  y_{t}\right) \\
&  =\lim_{t\rightarrow0}\left\{  \frac{\varphi\left(  y_{t}\right)  f\left(
y_{t}\right)  -\varphi\left(  y_{0}\right)  f\left(  y_{0}\right)  }%
{t}-\varphi\left(  y_{0}\right)  \frac{f\left(  y_{t}\right)  -f\left(
y_{0}\right)  }{t}\right\}
\end{align*}
exists and is equal to%
\[
\partial^{\lambda}\left(  \varphi f\right)  \left(  y_{0}\right)
-\varphi\left(  y_{0}\right)  \partial^{\lambda}f\left(  y_{0}\right)  .
\]
By Lemma \ref{Cancellation}, $\partial^{\lambda}\varphi\left(  y_{0}\right)
=\lim_{t\rightarrow0}\frac{\varphi\left(  y_{t}\right)  -\varphi\left(
y_{0}\right)  }{t}$ exists and
\begin{equation}
\partial^{\lambda}\varphi\left(  y_{0}\right)  \cdot f\left(  y_{0}\right)
=\partial^{\lambda}\left(  \varphi f\right)  \left(  y_{0}\right)
-\varphi\left(  y_{0}\right)  \partial^{\lambda}f\left(  y_{0}\right)  .
\label{EProductRule}%
\end{equation}
From the case $n =0$, we know that $\varphi$ is continuous on $Y$. Together
with the assumptions that $\partial^{\lambda}\left(  \varphi f\right)  $ and
$\partial^{\lambda}f$ are continuous and that $f$ is never zero, we can deduce
using (\ref{EProductRule}) and Lemma \ref{Cancellation} that $\partial
^{\lambda}\varphi$ is continuous. Since this is true for all multi-indices
$\lambda$ with $\left\vert \lambda\right\vert =1,$ we conclude that
$\varphi\in C^{1}\left(  X\right)  .$

Suppose that the lemma is true for some integer $n\geq1.$ Assume that
$f,\varphi f\in C^{n+1}\left(  Y,F\right)  ,$ with $f$ never zero on $Y.$ By
the inductive hypothesis$,$ $\varphi\in C^{n}\left(  Y\right)  .$ Also, for
any multi-index $\lambda$ with $\left\vert \lambda\right\vert =1,$ we have by
(\ref{EProductRule})
\[
\left(  \partial^{\lambda}\varphi\right) \cdot f=\partial^{\lambda}\left(
\varphi f\right)  -\varphi\partial^{\lambda}f.
\]
In particular, $\partial^{\lambda}\varphi\cdot f\in C^{n}\left(  Y,F\right)
$. By the inductive hypothesis, $\partial^{\lambda}\varphi\in C^{n}\left(
Y\right)  .$ Hence $\varphi\in C^{n+1}\left(  Y\right)  .$
\end{proof}

\begin{thm}
\label{thm8}Let $X$ and $Y$ be open subsets of ${\mathbb{R}}^{p}$ and
${\mathbb{R}}^{q}$ respectively, $p,q\in{\mathbb{N}}$, and let $E$ and $F$ be
Hausdorff topological vector spaces. Suppose that $m,n\in{\mathbb{N}}%
\cup\{0,\infty\}$ and $T:C^{m}(X,E)\rightarrow C^{n}(Y,F)$ is a vector space
isomorphism so that $T$ preserves common zeros. Then $p=q$ and $m=n$.
Moreover, there are a $C^{n}$-diffeomorphism $h:Y\rightarrow X$ and, for each
$y\in Y$, a vector space isomorphism $S_{y}:E\rightarrow F$ so that
\begin{equation}
Tf(y)=S_{y}(f\circ h(y))\text{ for all $f\in C^{n}(X,E)$ and all $y\in Y$.}
\label{Tf2}%
\end{equation}
Conversely, if a vector space isomorphism $T:C^{m}(X,E)\rightarrow C^{n}(Y,F)$
has the form (\ref{Tf2}), then $T$ preserves common zeros.
\end{thm}

\begin{proof}
Applying Theorem \ref{thm2} with $A\left(  X,E\right)  =C^{m}\left(
X,E\right)  $ and $A\left(  Y,F\right)  =C^{n}\left(  Y,F\right)  $, there
exist homeomorphisms $h:Z\rightarrow X$, $k:W\rightarrow Y$ so that $h\cup
k^{-1}:Z\cup Y\rightarrow X\cup W$ is a homeomorphism. According to Lemma
\ref{lem1}, $Z\subseteq Y$ and $W\subseteq X.$ Following the arguments as in
the proof of Theorem \ref{thm4}, $h:Y\rightarrow X\,$is a homeomorphism. Since
$X$ and $Y,$ which are open subsets of open subsets of ${\mathbb{R}}^{p}$ and
${\mathbb{R}}^{q}$ respectively, are homeomorphic, it follows from the Brouwer
Domain Invariance Theorem \cite[Chapter XVII, Theorem 3.1]{D} that $p=q.$

We now show that $h\in C^{n}\left(  Y,X\right)  .$ Let $v\in E\smallsetminus
\left\{  0\right\}  $ be fixed and set $f_{i}\in C^{m}\left(  X,E\right)  ,$
$i=1,\cdots,p,$ to be the function $f_{i}\left(  x\right)  =x_{i}v$ if
$x=\left(  x_{1},\cdots,x_{p}\right)  \in X.$ By Proposition \ref{Zx}, for all
$y\in Y$,
\begin{align*}
\left(  Tf_{i}\right)  \left(  y\right)   &  =T\mathbf{u}_{i}\left(  y\right)
,\text{ where }u_{i}=f_{i}\left(  h\left(  y\right)  \right)  \\
&  =(T\left(  h_{i}\left(  y\right)  \mathbf{v}\right)  )\left(  y\right)
,\text{ where }h\left(  y\right)  =\left(  h_{1}\left(  y\right)
,\cdots,h_{p}\left(  y\right)  \right)  ,\\
&  =h_{i}\left(  y\right)  T\mathbf{v}\left(  y\right)  .
\end{align*}
Since $T$ preserves common zeros$,$ $T\mathbf{v}$ is never zero. Applying
Lemma \ref{ProductRule} to the real-valued functions $h_{i}$ yields that
$h_{i}\in C^{n}\left(  Y\right)  .$ Hence $h=\left(  h_{1},\cdots
,h_{p}\right)  \in C^{n}\left(  Y,X\right)  .$ Similarly, $h^{-1}\in
C^{m}\left(  X,Y\right)  .$

Next, we show that $m=n.$ Given $\psi\in C^{m}\left(  Y\right)  ,$
$\varphi=\psi\circ h^{-1}\in C^{m}\left(  X\right)  .$ Fix $u\in
E\smallsetminus\left\{  0\right\}  ,~$and let $f=\varphi\cdot\mathbf{u}.$ Then
$f\in C^{m}\left(  X,E\right)  $ and hence $Tf\in C^{n}\left(  Y,F\right)  .$
By Proposition \ref{Zx},%
\begin{align*}
\left(  Tf\right)  \left(  y\right)   &  =T\mathbf{w}\left(  y\right)  ,\text{
where }w=f\left(  h\left(  y\right)  \right)  =\varphi\left(  h\left(
y\right)  \right)  u,\\
&  =T(\varphi\left(  h\left(  y\right)  \right)  \mathbf{u})\left(  y\right)
\\
&  =T(\psi\left(  y\right)  \mathbf{u})\left(  y\right) \\
&  =\psi\left(  y\right)  T\mathbf{u}\left(  y\right)  .
\end{align*}
Since $T\mathbf{u},Tf\in C^{n}\left(  Y,F\right)  $ and $T\mathbf{u}$ is never
$0$ on $Y,$ we conclude from Lemma \ref{ProductRule} that $\psi\in
C^{n}\left(  Y\right)  .$ Hence $C^{m}\left(  Y\right)  \subseteq C^{n}\left(
Y\right)  $ and thus $m\geq n.$ By symmetry, $m\leq n.$

For each $y\in Y,$ define $S_{y}:E\rightarrow F$ by%
\[
S_{y}\left(  u\right)  =\left(  T\mathbf{u}\right)  \left(  y\right)  .
\]
From the proof of Theorem \ref{thm4}, we see that $S_{y}$ is a vector space
isomorphism that satisfies (\ref{Tf2}). The converse is clear.
\end{proof}

\section{Automatic continuity}

In this section, we investigate the continuity properties of linear isomorphic
mappings between spaces of differentiable functions that preserves common
zeros. If $T:C^{m}\left(  X,E\right)  \rightarrow C^{n}\left(  Y,F\right)  $
is a linear isomorphism that preserves common zeros, where $X$ and $Y$ are
open subsets of $\mathbb{R}^{p}$ and $\mathbb{R}^{q}$ respectively, we have by
Theorem \ref{thm8} that $p=q$, $m=n$. Also, there are a $C^{m}$-diffeomorphism
$h:X\rightarrow Y$ and vector space isomorphisms $S_{y}:E\rightarrow F,y\in
Y$, satisfying (\ref{Tf2}). Define $J:C^{m}\left(  Y,F\right)  \rightarrow
C^{m}\left(  X,F\right)  $ by $\left(  Jg\right)  \left(  x\right)  =g\left(
h^{-1}\left(  x\right)  \right)  .$ Clearly $JT:C^{m}\left(  X,E\right)
\rightarrow C^{m}\left(  X,F\right)  $ is a vector space isomorphism
preserving common zeros and $\left(  JTf\right)  \left(  x\right)
=S_{x}\left(  f\left(  x\right)  \right)  $ for all $f\in C^{m}\left(
Y,F\right)  $ and all $x\in X.$ Therefore, in considering the continuity of
$T$ and the associated map $\Phi:Y\times E\rightarrow F,$ $\Phi\left(
y,u\right)  =S_{y}\left(  u\right)  ,$ there is no loss of generality in
assuming that $X=Y$ and that $h$ is the identity map.

\begin{prop}
\label{prop9} Let $X$ be an open subset of ${\mathbb{R}}^{p}$, $p
\in{\mathbb{N}}$, $E$ and $F$ be Hausdorff topological vector spaces and $m
\in{\mathbb{N}} \cup\{\infty\}$. Assume that $\Phi: X \times E \to F$ satisfies

\begin{enumerate}
\item For all $u \in E$, $\Phi(\cdot,u)$ belongs to $C^{m}(X,F)$;

\item For all $x \in X$, $\Phi(x, \cdot)$ is a linear operator from $E$ into
$F$;

\item \label{c} $\Phi$ is sequentially continuous.
\end{enumerate}

\noindent For all $f\in C^{m}(X,E)$, define $Tf(x)=\Phi(x,f(x)).$ Then $Tf\in
C^{m}(X,F)$ and $T$ is a linear operator from $C^{m}(X,E)$ to $C^{m}(X,F)$.
\end{prop}

\begin{proof}
The proposition holds for $m=\infty$ if it holds for all $m\in\mathbb{N}$. For
a fixed $m\in\mathbb{N}$ the proposition is a special case of the following
claim:\newline

\noindent\textbf{Claim.} If $f\in C^{m}(X,E)$ and $\left\vert \lambda
\right\vert \leq m,$ then $\theta_{\lambda}\left(  x\right)  =\Phi
(x,\partial^{\lambda}f(x))$ belongs to $C^{m-\left\vert \lambda\right\vert
}(X,F).$

We prove the claim by induction on $m-\left\vert \lambda\right\vert .$ Suppose
that $m-\left\vert \lambda\right\vert =0.$ If $f\in C^{m}(X,E)$, then
$\partial^{\lambda}f\in C(X,E)$. Thus $\partial^{\lambda}f\left(
x_{n}\right)  \rightarrow\partial^{\lambda}f\left(  x_{0}\right)  $ whenever
$x_{n}\rightarrow x_{0}.$ It follows from the sequential continuity of $\Phi$
that $\theta_{\lambda}\left(  x_{n}\right)  =\Phi(x_{n},\partial^{\lambda
}f(x_{n}))\rightarrow\Phi(x_{0},\partial^{\lambda}f(x_{0}))=\theta_{\lambda
}\left(  x_{0}\right)  .$ Hence $\theta_{\lambda}$ is sequentially continuous
on the metric space $X.$ Therefore, $\theta_{\lambda}\in C\left(  X,E\right)
.$

Suppose that the claim is true for $m-\left\vert \lambda\right\vert =k.$ Let
$m$ and $\lambda$ be such that $m-\left\vert \lambda\right\vert =k+1.$ If
$1\leq i\leq p,$ let $e_{i}$ denote the $i^{th}$ coordinate unit vector of
${\mathbb{R}}^{p}.$ For $t\neq0$ and $x_{0}\in X,$%
\begin{align*}
& \frac{\theta_{\lambda}\left(  x_{0}+te_{i}\right)  -\theta_{\lambda}\left(
x_{0}\right)  }{t}=\frac{\Phi\left(  x_{0}+te_{i},\partial^{\lambda}f\left(
x_{0}+te_{i}\right)  \right)  -\Phi\left(  x_{0},\partial^{\lambda}f\left(
x_{0}\right)  \right)  }{t}\\
& =\frac{\Phi\left(  x_{0}+te_{i},\partial^{\lambda}f\left(  x_{0}%
+te_{i}\right)  \right)  -\Phi\left(  x_{0}+te_{i},\partial^{\lambda}f\left(
x_{0}\right)  \right)  }{t}\\
& \quad\quad+\frac{\Phi\left(  x_{0}+te_{i},\partial^{\lambda}f\left(
x_{0}\right)  \right)  -\Phi\left(  x_{0},\partial^{\lambda}f\left(
x_{0}\right)  \right)  }{t}\\
& =\Phi\left(  x_{0}+te_{i},\frac{\partial^{\lambda}f\left(  x_{0}%
+te_{i}\right)  -\partial^{\lambda}f\left(  x_{0}\right)  }{t}\right)
+\frac{g\left(  x_{0}+te_{i}\right)  -g\left(  x_{0}\right)  }{t},
\end{align*}
where $g\left(  \cdot\right)  =\Phi(\cdot,\partial^{\lambda}f\left(
x_{0}\right)  )\in C^{m}(X,F),$ according to (1). If $t_{n}\neq0,$
$t_{n}\rightarrow0,$ then by (3),%
\[
\lim_{n\rightarrow\infty}\frac{\theta_{\lambda}\left(  x_{0}+t_{n}%
e_{i}\right)  -\theta_{\lambda}\left(  x_{0}\right)  }{t_{n}}=\Phi\left(
x_{0},\partial^{i}\partial^{\lambda}f\left(  x_{0}\right)  \right)
+\partial^{i}g\left(  x_{0}\right)  .
\]
Hence
\[
\partial^{i}\theta_{\lambda}\left(  x_{0}\right)  =\theta_{\lambda+e_{i}%
}\left(  x_{0}\right)  +\partial^{i}g\left(  x_{0}\right)  .
\]
Since $m-\left\vert \lambda+e_{i}\right\vert =m-\left\vert \lambda\right\vert
-1=k,$ $\theta_{\lambda+e_{i}}\in C^{k}(X,F)$ by induction$.$ Also,
$\partial^{i}g\in C^{m-1}(X,F)\subseteq C^{k}(X,F)$. It follows that
$\partial^{i}\theta_{\lambda}\in C^{k}(X,F)$ for all $1\leq i\leq p.$ Hence
$\theta_{\lambda}\in C^{k+1}(X,F)=C^{m-\left\vert \lambda\right\vert }(X,F).$
\end{proof}

Let $\left\vert x\right\vert $ denote the Euclidean norm of a vector
$x\in{\mathbb{R}}^{p}.$ A $C^{\infty}$-function $\varphi:$ ${\mathbb{R}}%
^{p}\rightarrow{\mathbb{R}}$ is called a $C^{\infty}$-\emph{bump} if
$0\leq\varphi\left(  x\right)  \leq1$ for all $x\in{\mathbb{R}}^{p},$
$\varphi\left(  x\right)  =1$ if $\left\vert x\right\vert \leq\frac{1}{2}$ and
$\varphi\left(  x\right)  =0$ if $\left\vert x\right\vert \geq1.$

\begin{prop}
\label{[B]}Let $E$ be a Hausdorff topological vector space. Suppose that
$\varphi$ is a $C^{\infty}$-bump on ${\mathbb{R}}^{p},$ $\left(  u_{n}\right)
$ is a bounded sequence in $E$ and $\left(  x_{n}\right)  $ is a sequence in
${\mathbb{R}}^{p}$ such that $|x_{n+1}-x_{0}|<\tfrac{1}{3}|x_{n}-x_{0}|<1$ for
all $n.$ Set $\varphi_{n}\left(  x\right)  =\varphi\left(  \frac{x-x_{n}%
}{r_{n}/2}\right)  ,$ where $r_{n}=\left\vert x_{n}-x_{0}\right\vert .$ If
$m\in\mathbb{N}$ and $\left(  c_{n}\right)  $ is a sequence of real numbers
such that
\[
\lim_{n\rightarrow\infty}\frac{c_{n}}{r_{n}^{m}}=0,
\]
then $f=\sum c_{n}\varphi_{n}\mathbf{u}_{n}\in C^{m}\left(  {\mathbb{R}}%
^{p},E\right)  .$
\end{prop}

\begin{proof}
For a given $n$ and any $a,b$ such that
\[
\frac{3r_{n+1}}{2}<a<\frac{r_{n}}{2}\text{ and }\frac{3r_{n}}{2}%
<b<\frac{r_{n-1}}{2}\text{ (}r_{0}=\infty\text{),}%
\]
set $A_{a,b}=\left\{  x\in{\mathbb{R}}^{p}:a<\left\vert x-x_{0}\right\vert
<b\right\}  .$ Since for all $k,$ $\varphi_{k}$ is supported on a ball
centered at $x_{k}$ with radius $r_{k}/2,$ we see that $f=c_{n}\varphi
_{n}\mathbf{u}_{n}$ on $A_{a,b}$. Therefore, $f$ is infinitely differentiable
on $A_{a,b}$ and
\[
\partial^{\lambda}f=c_{n}\partial^{\lambda}\varphi_{n}\cdot\mathbf{u}%
_{n}\text{ on }A_{a,b},\text{ }\left\vert \lambda\right\vert \leq m.
\]
It follows easily that $f$ is infinitely differentiable on ${\mathbb{R}}%
^{p}\smallsetminus\left\{  x_{0}\right\}  $ and
\begin{equation}
\partial^{\lambda}f=\sum_{n}c_{n}\partial^{\lambda}\varphi_{n}\cdot
\mathbf{u}_{n}\text{ on }{\mathbb{R}}^{p}\smallsetminus\left\{  x_{0}\right\}
,\text{ }\left\vert \lambda\right\vert \leq m. \label{PartialDerivativeatx_0}%
\end{equation}

\medskip

\noindent\textbf{Claim. }If\textbf{\ }$\left\vert \lambda\right\vert \leq m,$
then $\partial^{\lambda}f\left(  x_{0}\right)  =0.$

Indeed, if $\left\vert \lambda\right\vert =0,$ $f\left(  x_{0}\right)  =\sum
c_{n}\varphi_{n}\left(  x_{0}\right)  {u}_{n}=0.$ Suppose that the claim holds
for all $\lambda$ with $\left\vert \lambda\right\vert =k$ for some fixed
$k<m.$ Given any $\lambda,$ $\left\vert \lambda\right\vert =k-1$, $1\leq i\leq
p$ and $t\neq0,$%
\[
\frac{\partial^{\lambda}f\left(  x_{0}+te_{i}\right)  -\partial^{\lambda
}f\left(  x_{0}\right)  }{t}=\frac{\partial^{\lambda}f\left(  x_{0}%
+te_{i}\right)  }{t}%
\]
by the inductive hypothesis. By (\ref{PartialDerivativeatx_0}),
\[
\partial^{\lambda}f\left(  x_{0}+te_{i}\right)  =\left\{
\begin{array}
[c]{ccc}%
c_{n}\partial^{\lambda}\varphi_{n}\left(  x_{0}+te_{i}\right)  {u}_{n} &
\text{if} & \frac{r_{n}}{2}<\left\vert t\right\vert <\frac{3r_{n}}{2},\\
0 &  & \text{otherwise.}%
\end{array}
\right.
\]
Let $U$ be an open neighborhood of zero in $E.$ Since $\left(  u_{n}\right)  $
is a bounded sequence, there exists $\varepsilon>0$ such that $\left(  \alpha
u_{n}\right)  \subseteq U$ for all $\left\vert \alpha\right\vert
<\varepsilon.$ Choose $N$ such that $\frac{c_{n}}{\left(  r_{n}/2\right)
^{k}}\left\Vert \partial^{\lambda}\varphi\right\Vert _{\infty}<\varepsilon$
for all $n>N.$ Suppose that $0<\left\vert t\right\vert <\frac{3r_{N}}{2}$.
Then
\[
\frac{\partial^{\lambda}f\left(  x_{0}+te_{i}\right)  }{t}=\left\{
\begin{array}
[c]{ccc}%
\frac{c_{n}}{t}\partial^{\lambda}\varphi_{n}\left(  x_{0}+te_{i}\right)
{u}_{n} & \text{if} & \frac{r_{n}}{2}<\left\vert t\right\vert <\frac{3r_{n}%
}{2},n>N,\\
0 &  & \text{otherwise.}%
\end{array}
\right.
\]
But when $\frac{r_{n}}{2}<\left\vert t\right\vert <\frac{3r_{n}}{2}$ for some
$n>N,$%
\begin{align*}
\left\vert \frac{c_{n}\partial^{\lambda}\varphi_{n}\left(  x_{0}%
+te_{i}\right)  }{t}\right\vert  &  \leq\frac{c_{n}}{r_{n}/2}\frac{\left\Vert
\partial^{\lambda}\varphi\right\Vert _{\infty}}{\left(  r_{n}/2\right)
^{\left\vert \lambda\right\vert }}\\
&  =\frac{c_{n}}{\left(  r_{n}/2\right)  ^{k}}\left\Vert \partial^{\lambda
}\varphi\right\Vert _{\infty}<\varepsilon.
\end{align*}
Therefore,%
\[
\frac{\partial^{\lambda}f\left(  x_{0}+te_{i}\right)  -\partial^{\lambda
}f\left(  x_{0}\right)  }{t}\in U\text{ if }0<\left\vert t\right\vert
<\frac{3r_{N}}{2}.
\]
This shows that $\partial^{i}\partial^{\lambda}f\left(  x_{0}\right)  =0.$ So
the claim is verified by induction.\newline

For $\left\vert \lambda\right\vert =m$ and $x\in\mathbb{R}^{p},$
\[
\partial^{\lambda}f\left(  x\right)  =
\begin{cases}
c_{n}\partial^{\lambda}\varphi_{n}\left(  x\right)  u_{n} & = \quad
c_{n}\left(  \frac{2}{r_{n}}\right)  ^{\left\vert \lambda\right\vert }%
\partial^{\lambda}\varphi\left(  \frac{x-x_{n}}{r_{n}/2}\right) {u}_{n}\\
& \quad\text{if } \frac{r_{n}}{2}<\left\vert x-x_{n}\right\vert <\frac{3r_{n}%
}{2},\\
0 & \quad\text{otherwise}.
\end{cases}
\]
Since $\left(  u_{n}\right)  $ is bounded and $\lim\limits_{n\rightarrow
\infty}c_{n}\left(  \frac{2}{r_{n}}\right)  ^{\left\vert \lambda\right\vert
}\left\Vert \partial^{\lambda}\varphi\right\Vert _{\infty}=0,$ $\lim
\limits_{x\rightarrow x_{0}}\partial^{\lambda}f\left(  x\right)  =0.$ Hence
$\partial^{\lambda}f$ is continuous at $x_{0}.$ Thus $f\in C^{m}\left(
{\mathbb{R}}^{p},E\right)  .$
\end{proof}

Together with Proposition \ref{prop9}, the next result characterizes when a
map $T,$ defined in terms of the associated map $\Phi,$ sends functions from
$C^{m}\left(  X,E\right)  $ to $C^{m}\left(  X,F\right)  ,$ in the case where
$E$ and $F$ are locally convex and $E$ is metrizable.

\begin{thm}
\label{thm10}Let $X$ be an open subset of ${\mathbb{R}}^{p}$, $p\in
{\mathbb{N}}$. Suppose that $E$ and $F$ are Hausdorff topological vector
spaces and that $F$ is locally convex. If, for some $m\in{\mathbb{N}}$, the
map $\Phi:X\times E\rightarrow F$ has the property that $Tf(x)=\Phi(x,f(x))$
defines a linear operator $T$ from $C^{m}(X,E)$ to $C^{m}(X,F).$ Then $\Phi$
has the following properties.

\begin{enumerate}
\item For all $u\in E$, $\Phi(\cdot,u)$ belongs to $C^{m}(X,F)$;

\item For all $x\in X$, $\Phi(x,\cdot)$ is a linear operator from $E$ into $F$;

\item[(3$^{\prime}$)] \label{c'} If $(x_{n})$ is a sequence in $X$ converging
to some $x_{0}\in X$ and $(u_{n})$ is a bounded sequence in $E$, then
$(\Phi(x_{n},u_{n}))$ is a bounded sequence in $F$.
\end{enumerate}

\noindent Moreover, if $E$ is locally convex metrizable, then $\Phi$ is continuous.
\end{thm}

\begin{proof}
We first show that the ``moreover" statement follows from (1), (2), and
(3$^{\prime}$). Suppose that $E$ is locally convex metrizable and $\Phi$ is
not continuous. Then there exist sequences $\left(  x_{n}\right)  $ and
$\left(  u_{n}\right)  $ converging to $x_{0}$ and $u_{0}$ in $X$ and $E$
respectively such that $\Phi\left(  x_{n},u_{n}\right)  \nrightarrow
\Phi\left(  x_{0},u_{0}\right)  .$ It follows from the local convexity of $F$
that there is a continuous seminorm $\rho$ on $F$ such that%
\[
d_{n}=\rho\left(  \Phi\left(  x_{n},u_{n}\right)  ,\Phi\left(  x_{0}%
,u_{0}\right)  \right)  \nrightarrow0.
\]
Let $v_{n}=u_{n}-u_{0}.$ Then $\left(  v_{n}\right)  $ converges to $0.$ Note
that
\begin{align*}
\Phi\left(  x_{n},v_{n}\right)   &  -\Phi\left(  x_{0},0\right)  =\Phi\left(
x_{n},u_{n}\right)  -\Phi\left(  x_{n},u_{0}\right)  -\Phi\left(
x_{0},0\right) \\
&  =\left(  \Phi\left(  x_{n},u_{n}\right)  -\Phi\left(  x_{0},u_{0}\right)
\right)  +\left(  \Phi\left(  x_{0},u_{0}\right)  -\Phi\left(  x_{n}%
,u_{0}\right)  \right)  .
\end{align*}
By (1), $\lim\limits_{n\rightarrow\infty}\left(  \Phi\left(  x_{0}%
,u_{0}\right)  -\Phi\left(  x_{n},u_{0}\right)  \right)  =0.$ Therefore,
$\Phi\left(  x_{n},v_{n}\right)  \nrightarrow\Phi\left(  x_{0},0\right)  .$
This shows that we may assume $u_{0}=0$ without loss of generality.

By using a subsequence, we may further assume that for some $c>0,$ $d_{n}%
=\rho\left(  \Phi\left(  x_{n},u_{n}\right)  \right)  \geq c$ for all $n.$
Since $E$ is locally convex metrizable, there exists a sequence of continuous
seminorms $\left(  \rho_{k}\right)  $ on $E$ which determines the topology on
$E.$ For all $k,$ $\lim\limits_{n\rightarrow\infty}\rho_{k}\left(
u_{n}\right)  =0.$ Thus there exists a sequence $\left(  \alpha_{n}\right)  $
of positive numbers diverging to $\infty$ such that $\lim\limits_{n\rightarrow
\infty}\alpha_{n}\rho_{k}\left(  u_{n}\right)  =0$ for all $k.$ Now $\left(
\alpha_{n}u_{n}\right)  $ is a bounded sequence in $E.$ But $\rho\left(
\Phi\left(  x_{n},\alpha_{n}u_{n}\right)  \right)  =\alpha_{n}d_{n}\geq
\alpha_{n}c\rightarrow\infty.$ Hence (3$^{\prime}$) fails.

We now turn to proving that $\Phi$ satisfies (1), (2) and (3$^{\prime}$). The
first two parts are clear. Suppose that (3$^{\prime}$) fails. Then there exist
sequences $\left(  x_{n}\right)  $ and $\left(  u_{n}\right)  $ in $X$ and $E$
respectively, with $\left(  x_{n}\right)  $ converging to some $x_{0}\in X$,
$\left(  u_{n}\right)  $ bounded in $E,$ such that $(\Phi(x_{n},u_{n}))$ is
unbounded in $F$. Take a continuous seminorm $\rho$ on $F$ such that $\left(
\rho(\Phi(x_{n},u_{n}))\right)  $ is unbounded. Using a subsequence if
necessary, we may assume that
\[
3r_{n+1}=3\left\vert x_{n+1}-x_{0}\right\vert <\left\vert x_{n}-x_{0}%
\right\vert =r_{n}<1\text{ for all }n
\]
and that $\lim\limits_{n\rightarrow\infty}\rho(\Phi(x_{n},u_{n}))=\infty.$ Set
$c_{n}=\frac{1}{\rho(\Phi(x_{n},u_{n}))}r_{n}^{m}$ for all $n.$ By Proposition
\ref{[B]}, $f=\sum c_{n}\varphi_{n}\mathbf{u}_{n}\in C^{m}\left(  X,E\right)
.$ Therefore, $g(x)=\Phi(x,f(x))\in C^{m}\left(  X,F\right)  .$ By (2),
$g\left(  x\right)  =0$ whenever $f\left(  x\right)  =0.$ For any $n,$ $a,$
and $b$ such that
\[
\frac{3r_{n+1}}{2}<a<\frac{r_{n}}{2}\text{ and }\frac{3r_{n}}{2}%
<b<\frac{r_{n-1}}{2}\text{,}%
\]
$f=0$ on $A_{a,b}=\left\{  x\in X:a<\left\vert x-x_{0}\right\vert <b\right\}
.$ Thus $g=0$ on $A_{a,b}.$ By continuity, we have $\partial^\lambda g\left(  x_{0}\right)  =0$ if $|\lambda| \leq m.$ Let $i:F\rightarrow
\overline{F_{\rho}}$ be the natural map, where $F_{\rho}$ is the quotient
space $F/\rho^{-1}\left\{  0\right\}  $ normed by $\rho$ and $\overline
{F_{\rho}}$ is its completion. Then $G=i\circ g\in C^{m}(X,\overline{F_{\rho}%
}).$ By Taylor's Formula (see e.g., \cite[p.115]{L}), denoting by $\left(
x_{n}-x_{0}\right)  ^{\left(  m\right)  }$ the vector in $\left(
\mathbb{R}^{p}\right)  ^{m}$ with coordinates $x_{n}-x_{0}$ repeated
$m$-times,
\[
G\left(  x_{n}\right)  =G\left(  x_{0}\right)  +\sum_{k=1}^{m}\frac{D^{\left(
k\right)  }G\left(  x_{0}\right)  \left(  x_{n}-x_{0}\right)  ^{\left(
m\right)  }}{k!}+E_{n},
\]
where $\lim\limits_{n\rightarrow\infty}\frac{E_{n}}{\left\vert x_{n}%
-x_{0}\right\vert ^{m}}=0.$ Since $D^{\left(  k\right)  }g\left(
x_{0}\right)  =0$ for $0\leq k\leq m,$ we have $D^{\left(  k\right)  }G\left(
x_{0}\right)  =0$ for $0\leq k\leq m.$ Thus $\lim\limits_{n\rightarrow\infty
}\frac{G\left(  x_{n}\right)  }{\left\vert x_{n}-x_{0}\right\vert ^{m}}=0.$ It
follows that%
\begin{align*}
0  &  =\lim\limits_{n\rightarrow\infty}\frac{\rho\left(  g\left(
x_{n}\right)  \right)  }{\left\vert x_{n}-x_{0}\right\vert ^{m}}
=\lim\limits_{n\rightarrow\infty}\frac{\rho\left(  \Phi(x_{n},f(x_{n}%
))\right)  }{r_{n}^{m}}\\
&  =\lim\limits_{n\rightarrow\infty}\frac{\rho\left(  \Phi(x_{n},c_{n}%
u_{n})\right)  }{r_{n}^{m}}=\lim\limits_{n\rightarrow\infty}\frac{c_{n}%
\rho\left(  \Phi(x_{n},u_{n})\right)  }{r_{n}^{m}}=1,
\end{align*}
a contradiction.
\end{proof}

When $E$ and $F$ are locally convex Fr\'{e}chet spaces, Theorem \ref{thm10}
can be strengthened to yield continuity of the partial derivatives of $\Phi$
with respect to the coordinates of $x.$ The results also holds for $m=\infty$.
The assumption of completeness is crucial here as Examples \ref{Exm} and
\ref{Example_infinity} will show. The idea for Theorem \ref{thm11} comes from
\cite{A2-1}, especially \S 5.

\begin{thm}
\label{thm11} Let $X$ be an open subset of ${\mathbb{R}}^{p}$, $p\in
{\mathbb{N}},$ and let $m\in{\mathbb{N}}\cup\{\infty\}$. Assume that $E$ and
$F$ are locally convex Fr\'{e}chet spaces. Suppose that the map $\Phi:X\times
E\rightarrow F$ has the property that $Tf(x)=\Phi(x,f(x))$ defines a linear
operator $T$ from $C^{m}(X,E)$ into $C^{m}(X,F)$.

\begin{enumerate}
\item If $C^{m}(X,E)$ and $C^{m}(X,F)$ are endowed with complete linear metric topologies
that are stronger than the respective topologies of pointwise convergence,
then $T$ is continuous;

\item For any $\lambda$ with $|\lambda|<m+1$, the map $\Phi_{\lambda}:X\times
E\rightarrow F$ defined by $\Phi_{\lambda}(x,u)=(\partial^{\lambda}%
T\mathbf{u})(x)$ is continuous. (We adopt the convention $\infty+1=\infty$).
\end{enumerate}
\end{thm}

In the next two lemmas, $E$ and $F$ are locally convex Fr\'{e}chet spaces.

\begin{lem}
\label{c_n_infty} Assume that $\Phi$ and $T$ satisfy the hypotheses of Theorem
\ref{thm11} for some $m \in{\mathbb{N}} \cup\{\infty\}$. Let $\left(
u_{n}\right)  $ be a bounded sequence in $E$ and $\left(  x_{n}\right)  $ be a
sequence in ${\mathbb{R}}^{p}$ converging to a point $x_{0}$ such that
\[
3r_{n+1}<r_{n}<1\text{ for all }n,
\]
where $r_{n}=\left\vert x_{n}-x_{0}\right\vert .$ If $\left(  c_{n}\right)  $
is a sequence of real numbers such that
\begin{align*}
& \lim_{n\rightarrow\infty}\frac{c_{n}}{r_{n}^{m}}=0\text{ if $m
\in{\mathbb{N}}$ }\text{,}\\
& \lim_{n\rightarrow\infty}\frac{c_{n}}{r_{n}^{k}}=0\text{ for all }%
k\in\mathbb{N}\text{, if $m =\infty$},
\end{align*}
then
\begin{align*}
& \lim_{n\rightarrow\infty}\frac{c_{n}\Phi\left(  x_{n},u_{n}\right)  }%
{r_{n}^{m}}=0\text{ if $m \in{\mathbb{N}}$ },\\
& \lim_{n\rightarrow\infty}\frac{c_{n}\Phi\left(  x_{n},u_{n}\right)  }%
{r_{n}^{k}}=0\text{ for all }k\in\mathbb{N}\text{, if $m =\infty$}.
\end{align*}

\end{lem}

\begin{proof}
If $m \in{\mathbb{N}}$, the conclusion follows since $(\Phi(x_{n}, u_{n}))$ is
a bounded sequence by ($3^{\prime}$) of Theorem \ref{thm10}. Consider the case
$m = \infty$. Let $\varphi$ be a $C^{\infty}$-bump on ${\mathbb{R}}^{p}$ and
set $\varphi_{n}\left(  x\right)  =\varphi\left(  \frac{x-x_{n}}{r_{n}%
/2}\right) .$ According to Proposition \ref{[B]}, $f=\sum c_{n}\varphi
_{n}\mathbf{u}_{n}\in C^{\infty}\left(  X,E\right) $. Let $g\left(  x\right)
=Tf\left(  x\right)  =\Phi(x,f(x)).$ Then $g\in C^{\infty}\left(  X,F\right)
$. Using the same proof as in Theorem \ref{thm10}, we find that for any
continuous seminorm $\rho$ on $F$ and any $k\in{\mathbb{N}}$,%
\begin{align*}
0  &  =\lim\limits_{n\rightarrow\infty}\frac{\rho\left(  g\left(
x_{n}\right)  \right)  }{\left\vert x_{n}-x_{0}\right\vert ^{k}}
=\lim\limits_{n\rightarrow\infty}\frac{\rho\left(  \Phi(x_{n},f(x_{n}%
))\right)  }{r_{n}^{k}}\\
& =\lim\limits_{n\rightarrow\infty}\frac{\rho\left(  \Phi(x_{n},c_{n}%
u_{n})\right)  }{r_{n}^{k}}=\lim\limits_{n\rightarrow\infty}\rho\left(
\frac{c_{n}\Phi(x_{n},u_{n})}{r_{n}^{k}}\right)  .
\end{align*}
The lemma follows since the topology of $F$ is determined by continuous seminorms.
\end{proof}

\begin{lem}
\label{Continuousondense} Assume that $\Phi$ and $T$ satisfy the hypotheses of
Theorem \ref{thm11} for some $m \in{\mathbb{N}} \cup\{\infty\}$. Given any
compact set $K\subseteq X,$ there are only countably many $x\in K$ at which
$\Phi\left(  x,\cdot\right)  :E\rightarrow F$ is not continuous.
\end{lem}

\begin{proof}
Let $\rho$ be a continuous seminorm on $F$ and $q:F\rightarrow F_{\rho}$ be
the quotient map, where $F_{\rho}$ is the quotient space $F/\rho^{-1}\left\{
0\right\}  $ normed by $\rho.$ We claim that $q\circ\Phi\left(  x,\cdot
\right)  :E\rightarrow F_{\rho}$ is continuous for all but finitely many $x\in
K.$ Suppose, to the contrary, that there exists a sequence $\left(
x_{n}\right)  $ converging to $x_{0}$ in $K$ such that $q\circ\Phi\left(
x_{n},\cdot\right)  $ is discontinuous for all $n\in\mathbb{N}$. By taking a
subsequence if necessary, we may assume that $3r_{n+1}<r_{n}<1$ for all $n,$
where $r_{n}=\left\vert x_{n}-x_{0}\right\vert $. Let $c_{n}=r_{n}^{n}$ if $m
= \infty$ and $c_{n} = r_{n}^{m}$ if $m \in{\mathbb{N}}$. Suppose that the
topology on $E$ is determined by a sequence of seminorms $\left(  \rho
_{k}\right)  .$ For every $n,$ by the discontinuity of $q\circ\Phi\left(
x_{n},\cdot\right)  ,$ there exists $u_{n}\in E$ such that $\rho_{k}\left(
u_{n}\right)  \leq1,1\leq k\leq n,$ and
\[
\rho\left(  \Phi\left(  x_{n},u_{n}\right)  \right)  \geq\frac{1}{c_{n}}.
\label{LowerBddr(u_n)}%
\]
The sequence $\left(  u_{n}\right)  $ is bounded in $E.$ However,
\[
\lim\limits_{n\rightarrow\infty}\frac{c_{n}\rho\left(  \Phi(x_{n}%
,u_{n})\right)  }{r_{n}^{k}}\neq0
\]
for any $k \in{\mathbb{N}}$, contrary to Lemma \ref{c_n_infty}. This proves
the claim. Since the topology of $F$ is determined by countably many
seminorms, the lemma follows.
\end{proof}

\begin{proof}
[Proof of Theorem \ref{thm11}](1) It suffices to show that the graph of $T$ is
closed. Suppose that $f_{n}\rightarrow f$ in $C^{m}\left(  X,E\right)  $ and
$Tf_{n}\rightarrow g$ in $C^{m}\left(  X,F\right)  .$ Since the topologies on
$C^{m}\left(  X,E\right)  $ and $C^{m}\left(  X,F\right)  $ are stronger than
the topologies of pointwise convergence, $f_{n}\left(  x\right)  \rightarrow
f\left(  x\right)  $ and $Tf_{n}(x)\rightarrow g\left(  x\right)  $ for all
$x\in X.$ At any $x$ where $\Phi\left(  x,\cdot\right)  $ is continuous,
\[
Tf_{n}\left(  x\right)  =\Phi\left(  x,f_{n}\left(  x\right)  \right)
\rightarrow\Phi\left(  x,f\left(  x\right)  \right)  =Tf\left(  x\right)  .
\]
By Lemma \ref{Continuousondense}, $Tf_{n}\left(  x\right)  \rightarrow
Tf\left(  x\right)  $ for all $x\ $in a co-countable subset of $X$. Hence
$Tf=g$ on a dense subset of $X$. By continuity, $Tf=g.$

(2) Endow $C^{m}(X,E)$ and $C^{m}(X,F)$ with the respective topologies of
uniform convergence on compact sets of all partial derivatives of order
$<m+1$. Then the hypothesis of (1) is satisfied.

If $\left(  x_{n}\right)  $ and $\left(  u_{n}\right)  $ are sequences in $X$
and $E$ converging to $x_{0}\in X$ and $u_{0}\in E$ respectively, then the
sequence of constant functions $\left(  \mathbf{u}_{n}\right)  $ converges to
$\mathbf{u}_{0}$ in the given topology of $C^{m}\left(  X,E\right)  $. By
continuity of $T,$ $\left(  T\mathbf{u}_{n}\right)  $ converges to
$T\mathbf{u}_{0}$ in $C^{m}\left(  Y,F\right)  $. Since $K=\left\{
x_{n}\right\}  _{n=1}^{\infty}\cup\left\{  x_{0}\right\}  $ is compact,
$\partial^{\lambda}T\mathbf{u}_{n}\rightarrow\partial^{\lambda}T\mathbf{u}%
_{0}$ uniformly on $K$ for all $\lambda,$ $\left\vert \lambda\right\vert
<m+1.$ Thus $\lim\limits_{n\rightarrow\infty}\left(  \left(  \partial
^{\lambda}T\mathbf{u}_{n}\right)  \left(  x_{n}\right)  -\left(
\partial^{\lambda}T\mathbf{u}_{0}\right)  \left(  x_{n}\right)  \right)  =0.$
Hence
\begin{align*}
\Phi_{\lambda}\left(  x_{n},u_{n}\right)  -\Phi_{\lambda}\left(  x_{0}%
,u_{0}\right)   &  =\left(  \partial^{\lambda}T\mathbf{u}_{n}\right)  \left(
x_{n}\right)  -\left(  \partial^{\lambda}T\mathbf{u}_{0}\right)  \left(
x_{0}\right) \\
&  =\left(  \partial^{\lambda}T\mathbf{u}_{n}\right)  \left(  x_{n}\right)
-\left(  \partial^{\lambda}T\mathbf{u}_{0}\right)  \left(  x_{n}\right) \\
&  +\left(  \partial^{\lambda}T\mathbf{u}_{0}\right)  \left(  x_{n}\right)
-\left(  \partial^{\lambda}T\mathbf{u}_{0}\right)  \left(  x_{0}\right)
\end{align*}
converges to $0,$ keeping in mind the continuity of $\partial^{\lambda
}T\mathbf{u}_{0}.$
\end{proof}

For ease of reference, let us summarize the results from the preceding section
and the present one into a unified statement.

\begin{thm}
\label{Unified}Let $X$ and $Y$ be open subsets of ${\mathbb{R}}^{p}$ and
${\mathbb{R}}^{q}$ respectively, $p,q\in{\mathbb{N}}$, and let $E$ and $F$ be
Hausdorff topological vector spaces. Assume that $m,n\in{\mathbb{N}}%
\cup\{0,\infty\}$ and $T:C^{m}(X,E)\rightarrow C^{n}(Y,F)$ is a vector space
isomorphism so that $T$ preserves common zeros. Then $p=q$, $m=n$, and there
are a $C^{m}$-diffeomorphism $h:Y\rightarrow X$ and a map $\Phi:Y\times
E\rightarrow F$ such that

\begin{itemize}
\item[(a)] For all $u\in E,$ $\Phi\left(  \cdot,u\right)  $ belongs to
$C^{m}\left(  Y,F\right)  ;$

\item[(b)] For all $y\in Y,$ $\Phi\left(  y,\cdot\right)  $ is a vector space
isomorphism from $E$ onto $F;$

\item[(c)] $\left(  Tf\right)  \left(  y\right)  =\Phi\left(  y,f\left(
h\left(  y\right)  \right)  \right)  $ for all $f\in C^{m}\left(  X,E\right)
$ and all $y\in Y;$

\item[(d)] If $E$ and $F$ are locally convex metrizable and $m \in{\mathbb{N}%
}$, then $\Phi$ is continuous on $Y\times F.$ In particular, $E$ and $F$ are
isomorphic as topological vector spaces;

\item[(e)] If $E$ and $F$ and both locally convex Fr\'{e}chet, then the
conclusion in (d) also holds for $m=\infty.$ Furthermore, for any $\lambda$
with $\left\vert \lambda\right\vert <m+1,$ the map $\Phi_{\lambda}:Y\times
E\rightarrow F$ defined by $\Phi_{\lambda}\left(  y,u\right)  =\left(
\partial^{\lambda}T\mathbf{u}\right)  \left(  y\right)  $ is continuous, and
$T$ is continuous whenever $C^{m}\left(  X,E\right)  $ and $C^{m}\left(
Y,F\right)  $ are given complete linear metric topologies stronger than the respective
topologies of pointwise convergence.
\end{itemize}
\end{thm}

Example \ref{Exm} demonstrates that neither the continuity of $T$ nor that of
$\Phi_{\lambda},$ $\left\vert \lambda\right\vert >0,$ is guaranteed without
the completeness of $E$ and $F.$ Similarly, Example \ref{Example_infinity}
shows that the the continuity of $\Phi$ itself is not guaranteed if $m=\infty$
and $E$ and $F$ are not assumed to be complete. Denote by $c_{00}$ the space
of all finitely supported real sequences endowed with the sup-norm.

\begin{ex}
\label{Exm}For any $m\in{\mathbb{N}}$, there is a map $\Phi:{\mathbb{R}}\times
c_{00}\rightarrow c_{00}$ such that the induced map $T$ given by
$Tf(x)=\Phi(x,f(x))$ is a linear bijection from $C^{m}({\mathbb{R}},c_{00})$
onto itself. However, there is a sequence of functions $(f_{n})$ in
$C^{m}({\mathbb{R}},c_{00})$ so that $(f_{n}^{(k)})_{n}$ converges uniformly
to $0$ for all $k\leq m$, but $((Tf_{n})^{(k)}(0))_{n}$ does not converge to
$0$ for any $k$, $1\leq k\leq m$. Furthermore, the maps $\Phi_{k}(0,\cdot)$,
$1\leq k\leq m$, defined as in part (e) Theorem \ref{Unified}, are not bounded.
\end{ex}

If $f\in$ $C^{m}({\mathbb{R}},c_{00}),$ write $f=\sum f_{i}\mathbf{e}_{i},$
where $\left(  e_{i}\right)  $ is the unit vector basis of $c_{00}$ and
$f_{i}\in C^{m}({\mathbb{R)}}$.

\begin{lem}
\label{LemmaExample_inf_1}If $f\in$ $C^{m}({\mathbb{R}},c_{00}),$ then for all
$x_{0}\in{\mathbb{R}}\ $and $0\leq k\leq m,$ there exists $i_{0}=i_{0}\left(
x_{0}\right)  \in\mathbb{N}$ such that
\[
\lim_{x\rightarrow x_{0}}\sup_{i\geq i_{0}}\frac{|f_{i}^{\left(  k\right)
}\left(  x\right)  |}{\left\vert x-x_{0}\right\vert ^{m-k}}=0.
\]

\end{lem}

\begin{proof}
Let $x_{0}\in{\mathbb{R}}$ and $0\leq k\leq m$. Since $f\in C^{m}({\mathbb{R}%
},c_{00})$, for all $0\leq j\leq m,$
\begin{equation}
\lim_{x\rightarrow x_{0}}\sup_{i}|f_{i}^{\left(  j\right)  }\left(  x\right)
-f_{i}^{\left(  j\right)  }\left(  x_{0}\right)  |=0. \label{Ex1Eq1}%
\end{equation}
Since $f^{\left(  j\right)  }\left(  x_{0}\right)  \in c_{00}$ for all $j\leq
m,$ there exists $i_{0}\in\mathbb{N}$ such that $f_{i}^{\left(  j\right)
}\left(  x_{0}\right)  =0$ for all $i\geq i_{0}$ and $0\leq j\leq m.$ For
$i\geq i_{0},$ $x\neq x_{0},$ there exists, by Taylor's Theorem, $\xi
=\xi\left(  i,x\right)  \ $satisfying $0<\left\vert \xi-x_{0}\right\vert
<\left\vert x-x_{0}\right\vert ,$ such that
\begin{align*}
f_{i}^{\left(  k\right)  }\left(  x\right)   &  =f_{i}^{\left(  k\right)
}\left(  x_{0}\right)  +\cdots+\tfrac{f_{i}^{\left(  m-1\right)  }\left(
x_{0}\right)  }{\left(  m-1-k\right)  !}\left(  x-x_{0}\right)  ^{m-1-k}%
+\tfrac{f_{i}^{\left(  m\right)  }\left(  \xi\right)  }{\left(  m-k\right)
!}\left(  x-x_{0}\right)  ^{m-k}\\
&  =\frac{f_{i}^{\left(  m\right)  }\left(  \xi\right)  }{\left(  m-k\right)
!}\left(  x-x_{0}\right)  ^{m-k}.
\end{align*}
Thus%
\[
\frac{|f_{i}^{\left(  k\right)  }\left(  x\right)  |}{\left\vert
x-x_{0}\right\vert ^{m-k}}=\frac{|f_{i}^{\left(  m\right)  }\left(
\xi\right)  |}{\left(  m-k\right)  !}.
\]
By (\ref{Ex1Eq1}), for any $\varepsilon>0,$ there exists $\delta>0$ such that
\[
\sup_{0<\left\vert x-x_{0}\right\vert <\delta}\sup_{i\geq i_{0}}%
|f_{i}^{\left(  m\right)  }\left(  x\right)  |=\sup_{0<\left\vert
x-x_{0}\right\vert <\delta}\sup_{i\geq i_{0}}|f_{i}^{\left(  m\right)
}\left(  x\right)  -f_{i}^{\left(  m\right)  }\left(  x_{0}\right)
|<\varepsilon.
\]
Therefore, for $0<\left\vert x-x_{0}\right\vert <\delta,$
\[
\sup_{i\geq i_{0}}\frac{|f_{i}^{\left(  k\right)  }\left(  x\right)
|}{\left\vert x-x_{0}\right\vert ^{m-k}}<\frac{\varepsilon}{\left(
m-k\right)  !}.
\]

\end{proof}

\begin{lem}
\label{ConditionforCmMap}Let $\left(  \varphi_{i}\right)  $ be a sequence in
$C^{\infty}\left(  \mathbb{R}\right)  .$ If, for any $x_{0}\in\mathbb{R}$ and
$0\leq j\leq m,$ there exists $i_{0}\in\mathbb{N}$ such that
\[
\limsup_{x\rightarrow x_{0}}\sup_{i\geq i_{0}}|\varphi_{i}^{\left(  j\right)
}\left(  x\right)  ||x-x_{0}|^{j}=M<\infty,
\]
then $\sum f_{i}\varphi_{i}\mathbf{e}_{i}\in C^{m}\left(  \mathbb{R}%
,c_{00}\right)  $ whenever $\sum f_{i}\mathbf{e}_{i}\in C^{m}\left(
\mathbb{R},c_{00}\right)  .$
\end{lem}

\begin{proof}
Let $\varepsilon>0$ and $0\leq j,k\leq m$ with $j+k=m.$ According to the
hypothesis on $\left(  \varphi_{i}\right)  $ and Lemma
\ref{LemmaExample_inf_1}, there exists $\delta>0$ and $i_{1}\in\mathbb{N}$
such that
\[
|\varphi_{i}^{\left(  j\right)  }\left(  x\right)  ||x-x_{0}|^{j}<M+1\text{
and }\frac{|f_{i}^{\left(  k\right)  }\left(  x\right)  |}{\left\vert
x-x_{0}\right\vert ^{j}}<\varepsilon
\]
whenever $0<\left\vert x-x_{0}\right\vert <\delta$ and $i\geq i_{1}.$ Hence
\[
\lim_{x\rightarrow x_{0}}\sup_{i\geq i_{1}}|f_{i}^{\left(  k\right)  }\left(
x\right)  \varphi_{i}^{\left(  j\right)  }\left(  x\right)  |=0.
\]
We may also assume that $f_{i}^{\left(  \ell\right)  }\left(  x_{0}\right)
=0$ for $i\geq i_{1}$ and $0\leq\ell\leq m.$ Thus
\[
\lim_{x\rightarrow x_{0}}\sup_{i\geq i_{1}}|f_{i}^{\left(  k\right)  }\left(
x\right)  \varphi_{i}^{\left(  j\right)  }\left(  x\right)  -f_{i}^{\left(
k\right)  }\left(  x_{0}\right)  \varphi_{i}^{\left(  j\right)  }\left(
x_{0}\right)  |=0.
\]
Since $f_{i}^{\left(  k\right)  }\varphi_{i}^{\left(  j\right)  }$ is
continuous for $1\leq i<i_{1},$
\[
\lim_{x\rightarrow x_{0}}\sup_{i}|f_{i}^{\left(  k\right)  }\left(  x\right)
\varphi_{i}^{\left(  j\right)  }\left(  x\right)  -f_{i}^{\left(  k\right)
}\left(  x_{0}\right)  \varphi_{i}^{\left(  j\right)  }\left(  x_{0}\right)
|=0.
\]
It follows that $\sum\left(  f_{i}\varphi_{i}\right)  ^{\left(  m\right)
}\mathbf{e}_{i}\in C\left(  \mathbb{R},c_{00}\right)  .$
\end{proof}

\begin{proof}
[Proof of Example \ref{Exm}]Let $\varphi$ be a $C^{\infty}$-bump function. For
all $j\in\mathbb{N}\cup\left\{  0\right\}  ,$ let $C_{j}=\left\Vert
\varphi^{(j)}\right\Vert _{\infty}.$ Set $\varphi_{i}\left(  x\right)
=i\varphi\left(  ix\right)  x^{k}+2$ if $i\in\mathbb{N}_{k}=m\mathbb{N}%
+k,1\leq k\leq m.$ Note that
\begin{equation}
\left\vert \varphi\left(  ix\right)  \right\vert \leq\chi_{\left[  -\frac
{1}{i},\frac{1}{i}\right]  }\left(  x\right)  .\label{bump}%
\end{equation}
In particular $\varphi_{i}\left(  x\right)  \geq$ $1$ for all $x.$ Thus
$\psi_{i}\left(  x\right)  =\frac{1}{\varphi_{i}\left(  x\right)  }$ is
well-defined for $i\in\mathbb{N}$. We claim that for $x_{0}\in\mathbb{R}$ and
$0\leq j\leq m$, there exists $i_{0}\in\mathbb{N}$ such that

\begin{enumerate}
\item $\limsup\limits_{x\rightarrow x_{0}}\sup_{i\geq i_{0}}|\varphi
_{i}^{\left(  j\right)  }\left(  x\right)  ||x-x_{0}|^{j}<\infty,$

\item $\limsup\limits_{x\rightarrow x_{0}}\sup_{i\geq i_{0}}|\psi_{i}^{\left(
j\right)  }\left(  x\right)  ||x-x_{0}|^{j}<\infty.$
\end{enumerate}

If $x_{0}\neq0,$ there exists $i_{0}$ and a neighborhood $U$ of $x_{0}$ such
that $\varphi_{i}$ and $\psi_{i}$ are constant on $U$ for all $i>i_{0}.$ Hence
(1) and (2) hold. Assume that $x_{0}=0$. It follows from (\ref{bump}) that
$1\leq\left\vert \varphi_{i}\left(  x\right)  \right\vert \leq C_{0}\left\vert
x\right\vert ^{k-1}+2$ if $i\in\mathbb{N}_{k}$. Therefore (1) and (2) hold for
$j=0.$ Let $1\leq j\leq m$ be fixed. If $i\in\mathbb{N}_{k},$ then%

\[
\varphi_{i}^{\left(  j\right)  }\left(  x\right)  =i%
{\displaystyle\sum_{\ell=0}^{j\wedge k}}
\left(
\begin{array}
[c]{c}%
j\\
\ell
\end{array}
\right)  \left[  k\left(  k-1\right)  \cdots\left(  k-\ell+1\right)
x^{k-\ell}\right]  \left[  i^{j-\ell}\varphi^{\left(  j-\ell\right)  }\left(
ix\right)  \right]  .
\]
Thus%
\begin{align*}
|\varphi_{i}^{\left(  j\right)  }\left(  x\right)  |\left\vert x\right\vert
^{j}  &  \leq%
{\displaystyle\sum_{\ell=0}^{j\wedge k}}
\left(
\begin{array}
[c]{c}%
j\\
\ell
\end{array}
\right)  k^{\ell}\left\vert x\right\vert ^{k-1}\left\vert \varphi^{\left(
j-\ell\right)  }\left(  ix\right)  \right\vert \left\vert ix\right\vert
^{j-\ell+1}\\
&  \leq m^{m}\left\vert x\right\vert ^{k-1}%
{\displaystyle\sum_{\ell=0}^{j\wedge k}}
\left(
\begin{array}
[c]{c}%
j\\
\ell
\end{array}
\right)  C_{j-\ell}.
\end{align*}
The last inequality holds because $\varphi^{\left(  j-\ell\right)  }\left(
ix\right)  =0$ when $\left\vert ix\right\vert >1.$ This proves that (1) is
satisfied for $1\leq j\leq m.$

We prove (2) by induction on $j$. Assume that (2) is satisfied for all $j\leq
j_{0}$ where $j_{0}\leq m-1.$ Note that
\[
0=\left(  \varphi_{i}\psi_{i}\right)  ^{\left(  j_{0}+1\right)  }=%
{\displaystyle\sum_{\ell=0}^{j_{0}+1}}
\left(
\begin{array}
[c]{c}%
j_{0}+1\\
\ell
\end{array}
\right)  \varphi_{i}^{\left(  \ell\right)  }\psi_{i}^{\left(  j_{0}%
+1-\ell\right)  }.
\]
Therefore,%
\[
\psi_{i}^{\left(  j_{0}+1\right)  }\left(  x\right)  =\frac{-1}{\varphi
_{i}\left(  x\right)  }%
{\displaystyle\sum_{\ell=1}^{j_{0}+1}}
\left(
\begin{array}
[c]{c}%
j_{0}+1\\
\ell
\end{array}
\right)  \varphi_{i}^{\left(  \ell\right)  }\left(  x\right)  \psi
_{i}^{\left(  j_{0}+1-\ell\right)  }\left(  x\right)  .
\]
Since $\left\vert \frac{1}{\varphi_{i}\left(  x\right)  }\right\vert \leq1$,%
\[
|\psi_{i}^{\left(  j_{0}+1\right)  }\left(  x\right)  |\left\vert x\right\vert
^{j_{0}+1}\leq%
{\displaystyle\sum_{\ell=1}^{j_{0}+1}}
\left(
\begin{array}
[c]{c}%
j_{0}+1\\
\ell
\end{array}
\right)  |\varphi_{i}^{\left(  \ell\right)  }\left(  x\right)  ||x|^{\ell
}|\psi_{i}^{\left(  j_{0}+1-\ell\right)  }\left(  x\right)  ||x|^{j_{0}%
+1-\ell}.
\]
The desired inequality (2), with $j_{0}+1$ in place of $j,$ follows from (1)
and the inductive hypothesis.

By (1), (2) and Lemma \ref{ConditionforCmMap}, the map $T:C^{m}({\mathbb{R}%
},c_{00})\rightarrow C^{m}({\mathbb{R}},c_{00}),$ $\sum f_{i}\mathbf{e}%
_{i}\mapsto\sum\varphi_{i}f_{i}\mathbf{e}_{i}$ is a linear bijection. Finally,
the sequence $\left(  f_{n}\right)  =\left(  \frac{1}{n}\mathbf{e}_{n}\right)
$ in $C^{m}({\mathbb{R}},c_{00})$ and the sequences of derivatives $( f_{n}^{(
k) })_{n}$, $k\leq m$, converge to $0$ uniformly. However, for $1\leq k\leq
m,$ $\left(  Tf_{n}\right)  ^{\left(  k\right)  }\left(  0\right)  =k!e_{n}$
whenever $n\in\mathbb{N}_{k}.$ Hence $\lim\limits_{n\rightarrow\infty}\Phi
_{k}\left(  0,\frac{e_{n}}{n}\right)  =\lim\limits_{n\rightarrow\infty}%
Tf_{n}^{\left(  k\right)  }\left(  0\right)  \neq0.$ In particular, $\Phi
_{k}\left(  0,\cdot\right)  $ is not bounded.
\end{proof}

\begin{ex}
\label{Example_infinity}There is a map $\Phi:{\mathbb{R}}\times c_{00}%
\rightarrow c_{00}$ such that the induced map $T$ given by $Tf(x)=\Phi
(x,f(x))$ is a linear bijection from $C^{\infty}({\mathbb{R}},c_{00})$ onto
itself. However, $\Phi_{k}(0,\cdot)$ is not bounded for any $k\in{\mathbb{N}%
}\cup\{0\}$ and there is a sequence of functions $(f_{n})$ in $C^{\infty
}({\mathbb{R}},c_{00})$ so that $(f_{n}^{(k)})_{n}$ converges uniformly to $0$
for all $k\in{\mathbb{N}}\cup\{0\}$, but $((Tf_{n})^{(k)}(0))$ does not
converge to $0$ for any $k\in{\mathbb{N}}\cup\{0\}$.
\end{ex}

\begin{lem}
\label{ConditionforCinftyMap}Let $\left(  \varphi_{i}\right)  $ be a sequence
in $C^{\infty}\left(  \mathbb{R}\right)  .$ If for any $x_{0}\in\mathbb{R}$
and $k\in\mathbb{N}\cup\left\{  0\right\}  ,$ there exists $i_{0}=i_{0}\left(
k\right)  \in\mathbb{N}$ such that
\[
\limsup\limits_{x\rightarrow x_{0}}\sup_{i\geq i_{0}}|\varphi_{i}^{\left(
k\right)  }\left(  x\right)  |\left\vert x-x_{0}\right\vert ^{k+1}<\infty,
\]
then $\sum f_{i}\varphi_{i}\mathbf{e}_{i}\in C^{\infty}\left(  \mathbb{R}%
,c_{00}\right)  $ whenever $\sum f_{i}\mathbf{e}_{i}\in C^{\infty}\left(
\mathbb{R},c_{00}\right)  .$
\end{lem}

\begin{proof}
Given $x_{0}\in\mathbb{R}$, $k,j\in\mathbb{N}\cup\left\{  0\right\}  ,$ it
follows from Lemma \ref{LemmaExample_inf_1} that there exists $i_{1}%
=i_{1}\left(  x_{0},k,j\right)  >i_{0}\left(  j\right)  $ such that
$f_{i}^{\left(  k\right)  }\left(  x_{0}\right)  =0$ if $i\geq i_{0}$ and
\[
\lim_{x\rightarrow x_{0}}\sup_{i\geq i_{1}}\frac{|f_{i}^{\left(  k\right)
}\left(  x\right)  |}{\left\vert x-x_{0}\right\vert ^{j+1}}=0.
\]
It follows readily that $\lim_{x\rightarrow x_{0}}\sup_{i\geq i_{1}}%
|f_{i}^{\left(  k\right)  }\left(  x\right)  \varphi_{i}^{\left(  j\right)
}\left(  x\right)  |=0.$ The remainder of the argument proceeds as in the
proof of Lemma \ref{ConditionforCmMap}.
\end{proof}

\begin{proof}
[Proof of Example \ref{Example_infinity}]Let $\varphi$ be a $C^{\infty}$-bump
function$.$ For all $j\in\mathbb{N}\cup\left\{  0\right\}  ,$ let
$C_{j}=\left\Vert \varphi^{(j)}\right\Vert _{\infty}.$ Let $\left(
\mathbb{N}_{k}\right)  _{k\in\mathbb{N}}$ be a partition of $\mathbb{N}$ into
infinite sets such that $i\geq k$ if $i\in\mathbb{N}_{k}.$ Set $\varphi
_{i}\left(  x\right)  =i\varphi\left(  ix\right)  x^{k-1}+2,$ if
$i\in\mathbb{N}_{k}$ and $\psi_{i}\left(  x\right)  =\frac{1}{\varphi
_{i}\left(  x\right)  }$. By arguments similar to those used in the proof of
Example \ref{Exm}, one can show that for all $x_{0}\in\mathbb{R}$ and
$j\in\mathbb{N\cup}\left\{  0\right\}  $, there exists $i_{0}\in\mathbb{N}$
such that

\begin{enumerate}
\item $\lim_{x\rightarrow x_{0}}\sup_{i\geq i_{0}}|\varphi_{i}^{\left(
j\right)  }\left(  x\right)  |\left\vert x-x_{0}\right\vert ^{j+1}<\infty,$

\item $\lim_{x\rightarrow x_{0}}\sup_{i\geq i_{0}}|\psi_{i}^{\left(  j\right)
}\left(  x\right)  |\left\vert x-x_{0}\right\vert ^{2j+1}<\infty.$
\end{enumerate}

It follows by Lemma \ref{ConditionforCinftyMap} that the map $T:C^{\infty
}({\mathbb{R}},c_{00})\rightarrow C^{\infty}({\mathbb{R}},c_{00}),$ $\sum
f_{i}\mathbf{e}_{i}\mapsto\sum\varphi_{i}f_{i}\mathbf{e}_{i},$ is a linear
bijection. For each $n,$ let $f_{n}=\frac{\mathbf{e}_{n}}{n}\in C^{\infty
}({\mathbb{R}},c_{00})$. Then $\left(  f_{n}^{\left(  k\right)  }\right)  $
converges uniformly to zero for all $k\in\mathbb{N}\cup\left\{  0\right\}  .$
However, $Tf_{n}\left(  x\right)  =\dfrac{\varphi_{n}\left(  x\right)
\mathbf{e}_{n}}{n}=\varphi\left(  nx\right)  x^{k-1}+\frac{2}{n}$ for all
$n\in\mathbb{N}_{k}.$ $\left\vert x\right\vert <1$. Recall that $\varphi
_{n}\left(  x\right)  =1$ if $\left\vert x\right\vert \leq\frac{1}{2n}.$
Therefore, for all $k\in\mathbb{N\cup}\left\{  0\right\}  $ and $n\in
\mathbb{N}_{k+1}$%
\[
\left(  Tf_{n}\right)  ^{\left(  k\right)  }\left(  0\right)  =\left\{
\begin{array}
[c]{ccc}%
1+2/n & \text{if} & k=0,\\
k! & \text{if} & k\geq1.
\end{array}
\right.
\]
Hence $\left(  Tf_{n}\right)  ^{\left(  k\right)  }\left(  0\right)  \ $does
not converge to $0$ for all $k.$ Since $\Phi_{k}\left(  0,\frac{e_{n}}%
{n}\right)  =\left(  Tf_{n}\right)  ^{\left(  k\right)  }\left(  0\right)  ,$
we also see that $\Phi_{k}\left(  0,\cdot\right)  $ is not bounded for any
$k\in\mathbb{N\cup}\left\{  0\right\}  .$
\end{proof}

\end{document}